\documentclass[11pt]{aims}
\usepackage{amsmath, amssymb, mathrsfs}
  \usepackage{paralist}
  \usepackage{graphics} 
  \usepackage{graphicx}
  \usepackage{epsfig} 
\usepackage{graphicx}  \usepackage{epstopdf}
 \usepackage[colorlinks=true]{hyperref}
 \hypersetup{urlcolor=blue, citecolor=red}
 \usepackage[toc,page]{appendix}
\usepackage{chngcntr}
\usepackage{nicefrac}
\usepackage[square,numbers]{natbib}

\usepackage{titlesec}
\titleformat{\section}{\Large\bfseries}{\thesection.}{4pt}{}
\titleformat{\subsection}{\large\bfseries}{\thesection.\arabic{subsection}.}{4pt}{}
\titleformat{\subsubsection}{\bfseries}{\thesection.\arabic{subsection}.\arabic{subsubsection}.}{4pt}{}
\titleformat*{\paragraph}{\bfseries}
\titleformat*{\subparagraph}{\bfseries}
\usepackage[margin=1.2in]{geometry}

  \def\currentyear{}
   \def\currentmonth{}
    \def\ppages{}
     \def\DOI{}

\newtheorem{theorem}{Theorem}[section]
\newtheorem{corollary}[theorem]{Corollary}

\newtheorem{lemma}[theorem]{Lemma}
\newtheorem{proposition}[theorem]{Proposition}
\theoremstyle{definition}
\newtheorem{definition}[theorem]{Definition}
\newtheorem{remark}[theorem]{Remark}

\newcommand{\ep}{\varepsilon}

\newcommand{\Rb}{\mathbb{R}}

\newcommand{\Lc}{\mathcal{L}}

\newcommand{\Sc}{\mathcal{S}}

\newcommand{\Oc}{\mathcal{O}}

\newcommand{\Cc}{\mathcal{C}}

\newcommand{\Fc}{\mathcal{F}}

\newcommand{\Pc}{\mathcal{P}}

\newcommand{\Ls}{\mathscr{L}}

\newcommand{\Ms}{\mathscr{M}}

\newcommand{\As}{\mathscr{A}}

\newcommand{\pa}{\partial}

\newcommand{\inn}{\textup{in}}

\newcommand{\bd}{\textup{bou}}

\newcommand{\be}{\begin{equation}}
\newcommand{\ee}{\end{equation}}

\newcommand{\intb}{\int_{-\nicefrac{1}{\nu}}^\infty}

\numberwithin{equation}{section}

\title[Collapsing-ring blowup solutions for the Keller-Segel system] 
      {Collapsing-ring blowup solutions for the Keller-Segel system in three dimensions and higher}

\author[]{}

\author[C. Collot]{Charles Collot}
\address{CNRS and CY Cergy Paris Universit\'e, 2 rue Adolphe Chauvin, 95300 Pontoise, France.}
\email{ccollot@cyu.fr}
\author[T. Ghoul]{Tej-Eddine Ghoul}
\address{Department of Mathematics, New York University in Abu Dhabi, Saadiyat Island, P.O. Box 129188, Abu Dhabi, United Arab Emirates.}
\email{teg6@nyu.edu}
\author[N. Masmoudi]{Nader Masmoudi}
\address{Department of Mathematics, New York University in Abu Dhabi, Saadiyat Island, P.O. Box 129188, Abu Dhabi, United Arab Emirates.}
\email{masmoudi@cims.nyu.edu}
\author[V. T. Nguyen]{Van Tien Nguyen}
\address{Department of Mathematics, New York University in Abu Dhabi, Saadiyat Island, P.O. Box 129188, Abu Dhabi, United Arab Emirates.}
\email{Tien.Nguyen@nyu.edu}

\thanks{ 
\today}
\begin{document}
\maketitle


\begin{abstract} 

We consider the parabolic-elliptic Keller-Segel system in three dimensions and higher, corresponding to the mass supercritical case. We construct rigorously a solution which blows up in finite time by having its mass concentrating near a ring that shrinks to a point. The singularity is  in particular of type II, non self-similar and resembles a traveling wave imploding at the origin in renormalized variables. We show the stability of this dynamics among spherically symmetric solutions, and to our knowledge, this is the first stability result for such phenomenon for an evolution PDE. We develop a framework to handle the interactions between the two blowup zones contributing to the mechanism: a thin inner zone around the ring where viscosity effects occur, and an outer zone where the evolution is mostly inviscid.
\end{abstract}

\section{Introduction}

\subsection{Singularity formation for the parabolic-elliptic Keller-Segel system}

This paper is concerned with the parabolic-elliptic Keller-Segel system
\begin{equation}\label{3Dsys}
\left\{\begin{array}{rl}
 \partial_t u = &\nabla \cdot(\nabla u - u \nabla \Phi_u),\\
 -\Delta \Phi_u& = u,
\end{array}
 \right. \quad \textup{in}\;\; \Rb^d.
\end{equation}
Solutions might form singularities in finite time. This is relevant in the perspective of understanding the qualitative behavior of solutions to \eqref{3Dsys}, what we describe now. This is also interesting in regard of singularity formation for other equations, and we analyze our result in this broader context in the comments after Theorem \ref{theo:1}.\\

\noindent System \eqref{3Dsys} arises in modelling biological chemotaxis processes and stellar dynamics. Here, $u(x,t)$ stands for the density of particles or cells and $\Phi_u$ is a self-interaction potential. We refer to \cite{KSjtb70}, \cite{KSjtb71a} \cite{KSjtb71b} for a derivation of a general formulation of \eqref{3Dsys} to describe the aggregation of the slime mold amoebae Dictyostelium discoideum and \cite{Warma92}, \cite{Wjam92} for the case $d=3$ as a model of stellar dynamics under friction and fluctuations. We recommend the reference \cite{Hjdmv03} where the author gives a nice survey of mathematical problems encountered in the study of \eqref{3Dsys} and a wide bibliography including references of related models. 

We recall that from standard argument, given a radial function $u_0\in L^\infty(\mathbb R^d)$, there exists a unique local in time solution to \eqref{3Dsys}, see \cite{GMSarma11} for example. We refer to \cite{B} for further results on local well-posedness in other spaces. Moreover, by a comparison argument, if $u$ blows up in finite time $T>0$, there holds the lower bound of the blowup rate
$$
\| u (t)\|_{L^\infty(\mathbb R)}\geq (T-t)^{-1}
$$
(see \cite{Ko} for other lower bounds). It is well known that the solution exists globally in time for $d=1$, see \cite{Namsa95}. The case $d=2$ is called $L^1$-critical in the sense that  the scale transformation 
\be \label{id:scaling}
\forall \lambda > 0, \quad  u_\lambda(x,t) = \frac{1}{\lambda^2}u\Big(\frac{x}{\lambda}, \frac{t}{\lambda^2} \Big), 
\ee
preserves the total mass $M = \int_{\Rb^2} u(x,t)dx = \int_{\Rb^2}u_\lambda(x,t) dx$ which is a conserved quantity for \eqref{3Dsys}. There exhibits a remarkable dichotomy:\\
- If $M < 8\pi$, Dolbeault-Perthame \cite{DPcrasp04} proved that the solution is global in time. This result was further completed and improved in \cite{BDPjde06}. The main ingredient in deriving the sharp threshold $8\pi$ for global existence is the use of the free-energy functional
\begin{equation}\label{def:Fcu}
\Fc[u](t) = 	\int_{\Rb^d}  u(x,t) \Big[\log u(x,t)  - \frac{1}{2}\Phi_u(x,t)\Big] dx,
\end{equation}
combined with the logarithmic Hardy-Littlewood-Sobolev inequality 
$$\int f(x) \log f(x) dx  + \frac{2}{M_f} \int_{\Rb^2}\int_{\Rb^2} f(x)f(y) \log|x-y| dx dy \geq - M_f(1 + \log \pi - \log M_f),$$
where $M_f = \int_{\Rb^2} f(x) dx.$\\
- If $M = 8\pi$ and the second moment is finite, i.e. $\int_{\Rb^2} |x|^2 u(x,t) dx < +\infty$, Blanchet-Carrillo-Masmoudi \cite{BCMcpam08} showed the existence
of infinite time blowup solutions to \eqref{3Dsys}. Again, the free-energy functional $\Fc$ played a crucial role in the work \cite{BCMcpam08}. Concrete examples have been constructed in \cite{GMcpam18} and \cite{DPDMWarx19} where the authors rectified the blowup law obtained in \cite{Snon07}: 
$$\|u(t)\|_{L^\infty(\Rb^2)} \sim c \ln t \quad \textup{as}\;\; t \to +\infty.$$
Certain solutions with infinite second moments converge to a fixed stationary state \cite{BCCjfa12}, with quantitative rates \cite{CaFi}.

- If $M > 8\pi$, any positive solution blows up in finite time. Indeed, the equation for the second moment
$$\frac{d}{dt}\int_{\Rb^2} |x|^2 u(x,t)dx =  4M \left( 1 - \frac{M}{8\pi}\right),$$
cannot be satisfied for all times as the right-hand side is strictly negative and the differentiated quantity is positive. Finite time blowup solutions had been predicted in \cite{Njtb73}, \cite{CPmb81}, \cite{JLtams92}. Rigorous constructions were later done by Herrero-Vel{\'a}zquez, \cite{HVma96}, \cite{Vsiam02}, Rapha\"el-Schweyer \cite{RSma14} and the present authors \cite{CGNNcpam21} where  the following blowup dynamics was confirmed:
\begin{equation}\label{stable2D}
u(t) \approx \frac{1}{\lambda^{2}(t)} U\Big( \frac{x}{\lambda(t)} \Big) \quad \textup{with}\;\; \lambda(t) = 2e^{-\frac{2+ \gamma}{2}}\sqrt{T-t}e^{-\sqrt{\frac{|\log (T-t)|}{2}}} (1 + o_{t \uparrow T}(1)),
\end{equation}
where $U(x) = 8(1 + |x|^2)^{-2}$ is stationary and satisfies $\int_{\Rb^2} U(x)dx = 8\pi$. This blowup dynamics is stable and is believed to be generic thanks to the partial classification result of Mizoguchi \cite{Mcpam20} who proved that \eqref{stable2D} is the only blowup mechanism that occurs among radial nonnegative solutions. Other blowup rates corresponding to unstable blowup dynamics were also obtained in \cite{CGNNcpam21} as a consequence of a detailed spectral analysis obtained in \cite{CGNNarx19a}.  

\bigskip

\noindent The case $d \geq 3$ is quite different from $d = 2$. The system is called mass-supercritical, and the scaling transformation \eqref{id:scaling} preserves the $L^{d/2}$-norm: $\|u_\lambda(0)\|_{L^{d/2}(\Rb^d)} = \|u(0)\|_{L^{d/2}(\Rb^d)}$. There is a critical threshold on $\| u(0)\|_{L^{d/2}}$ that distinguishes between the global existence and finite time blowup. In particular, the authors of \cite{CCEcpde12} showed that for initial data $\|u(0)\|_{L^{d/2}} < C(d)$, where $C(d)$ is related to the Gagliardo-Nirenberg inequality\footnote{Namely, $C(d)=\frac{8}{d} C_{GN}^{-2(1 + 2/d)}(\frac d2,d)$ where $C_{GN}$ is the Gagliardo-Nirenberg inequality's constant $\|v\|_{L^{ \frac{2(p+1)}{p} }} \leq C_{GN}(p,d) \|\nabla v\|_{L^2}^\frac{d}{2(p+1)} \, \|v\|_{L^2}^{1 - \frac{d}{2(p+1)}}$.}, the (weak) solution is global in time. See also \cite{CPZmjm04} and references therein for earlier results concerning the global existence for \eqref{3Dsys}. It is known that there exist finite-time blowup solutions to \eqref{3Dsys}, depending on the initial size of the solution, see for example \cite{BKZnonl15}, \cite{CPZmjm04}. Since the total mass is conserved for the solution of \eqref{3Dsys}, note then that in contrast with the two-dimensional case, solutions can blow up with any arbitrary mass 
thanks to the relation $M(u_\lambda(0)) = \lambda^{d-2}M(u(0)).$
 
We say that $u$ exhibits Type I blowup if there is a constant $C > 0$ such that 
$$\limsup_{t \to T} (T-t)\| u(t)\|_{L^\infty(\Rb^d)} < C,$$
otherwise, the blowup is said to be of Type II. This notion is motivated by the ODE $u_t = u^2$ obtained by discarding
diffusion and transport in the equivalent equation $u_t = \Delta u + u^2 - \nabla \Phi_u. \nabla u$ to the first one in \eqref{3Dsys}.

The trace at blowup time $u(T,x)=\lim_{t\uparrow T}u(t,x)$ of Type I blowup solutions is proved to be asymptotically self-similar near the origin by Giga-Mizoguchi-Senba \cite{GMSarma11}, for a class of radial nonnegative solutions. A countable family of type I self-similar blowup solutions was obtained in \cite{HMVjcam98} (see also \cite{Se}). Type II blowup solutions were formally constructed by Herrero-Medina-Vel\'azquez \cite{HMVnonl97} in the radially symmetric setting for $d = 3$. We recommend \cite{BCKSVnon99} for a nice survey and numerical observations for singularity formation in three dimensions. No rigorous detailed construction and stability analysis of blowup solutions have been known, prior to the present work, for the case $d \geq 3$.

In this paper, we construct type II finite-time blowup solutions to \eqref{3Dsys} making rigorous the formal argument of \cite{HMVnonl97}. A part of the mass of the solution is concentrated around a ring that collapses to the origin. Our result is for spherically symmetric solutions for which we show the stability of the dynamics. We introduce the profile
$$
W(\xi)=\frac 18 \cosh^{-2}\left(\frac \xi 4 \right).
$$

\begin{theorem}[Existence and stability of a collapsing-ring blowup solution to \eqref{3Dsys}]  \label{theo:1}  For any $d \geq 3$, there exists an open set of spherically symmetric functions $\Oc \subset L^{\infty}(\mathbb R^d)$ such that for any $ u_0 \in \Oc$, the solution $u$ to \eqref{3Dsys} with initial data $u(0)=u_0$ blows up with type II at time $T(u_0)>0$ and can be decomposed as 
\begin{equation}\label{eq:decomp_murt}
u(x,t) = \frac{M(t)}{R^{d-1}(t)\lambda(t)} \left[W\left(\frac{|x|-R(t)}{\lambda(t)}  \right)+\tilde u\left(x,t\right)\right],
\end{equation}
 where
\begin{equation}\label{eq:bdparametersmainth}
\lambda(t) = \frac{R(t)^{d-1}}{M(t)}, \quad M(t) =  M_\infty \big(1 + o_{t \uparrow T}(1)\big), \quad R(t) = c_d M_\infty^{\frac 1d} (T-t)^\frac{1}{d} \big(1 + o_{t \uparrow T}(1)\big),
\end{equation}
with $c_d=(\frac d2)^{\frac 1d}$ and $M_\infty (u_0)>0$, and 
\begin{equation}\label{eq:bdvarepsilonmainth}
 \| \tilde u(t)\|_{L^\infty(\mathbb R^d)} \to 0 \quad \textup{as}\;\; t \to T.
\end{equation}
Moreover, the functions $T:u_0\mapsto T(u_0)$ and $M_\infty :u_0\mapsto M_\infty(u_0)$ are continuous on $\mathcal O$.

\end{theorem}  

\vspace*{0.1cm}

\begin{remark}
The collapsing ring is located at the distance $R(t)$ from the origin and has the width $\lambda (t)$. The total mass carried around the ring is $|\mathbb S^{d-1}|M(t)$, where $|\mathbb S^{d-1}|$ is the surface measure of the unit sphere in $\mathbb R^d$.
\end{remark}

\begin{remark} A detailed description of the open set $\Oc$ is given in Section \ref{sec:bootstrap}. We suspect the solution to be unstable by nonradial perturbations.
 \end{remark}

\noindent \textbf{Comments}: 
\noindent \emph{(i) Ring blowups among type II blowups and their stability.} For a general evolution equation, during a self-similar blowup all terms in the equation contribute with equal strength to the singularity. In contrast, during type II blowup as defined for most parabolic equations, a norm of the solution does not diverge according to the self-similar rate, which formally means that $\pa_t u$ is subleading as $t\uparrow T$. In the two-dimensional blowup for the Keller-Segel equation, $\pa_t u $ is fully subleading \cite{HVma96,RSma14,CGNNcpam21}, so that the profile is a stationary state. This is the most studied situation among type II blowup for evolution PDEs, see e.g. \cite{KST,RaRo,RoSt,DDW}.

The only known type II blowups where $\pa_t$ is subleading, but only after a space translation, in which case the blowup profile is a traveling wave, are the following. The seminal work \cite{MaMe} concerned the critical gKdV equation. A one-dimensional traveling wave was embedded in higher dimensions to produce a ring blowup for NLS in \cite{MRSduke14} (observed numerically in \cite{FGWphys05,FGWphys07}). However, the construction was based on a compactness argument specific to time-reversible equations, that bypasses the stability analysis and cannot be used here.

The present work gives then for the first time a stability result for a ring-blowup solution involving a traveling wave. Note that any type II blowup involves two blowup zones contributing to the singularity: an \emph{inner zone} close to the blowup profile in which $\pa_t$ is negligible, and an \emph{outer zone} where $\pa_t$ is not negligible anymore, close to the tail of the profile. The new challenges, in comparison with the most studied situation of type II blowups involving stationary states, are the following. As the profile is a traveling wave, both equations in the inner and outer zones involve transport terms, they not only change each equations separately, but also the way the two zones interact. In addition, since the original equation is only approximately one-dimensional near the ring, this generates error terms in the inner zone, while the outer zone is truly $d$-dimensional. Finally, a particularity of the present situation is that the dynamics in the outer zone is  inviscid to leading order. The novelties of our analysis to deal with these issues are explained in the strategy of the proof below. \\

\noindent \emph{(ii) Link with the Burgers equation.} In the renormalized partial mass variables for the inner zone around the ring (see \eqref{strat:def:murt} and \eqref{strat:def:xis}), the profile is $Q$ given by \eqref{strat:eq:murt} which is the traveling wave of the Burgers equation
\begin{equation} \label{eq:Burgers}
\pa_s f = \pa_\xi^2 f + \frac{1}{2}\pa_\xi(f^2).
\end{equation}
The stability of $Q$ for \eqref{eq:Burgers} was studied in \cite{IlOl,Sam76,Njjam85,KM}. The Cole-Hopf transform of \cite{Njjam85} however cannot be applied here, and the spectral method developed in \cite{Sam76} can only handle exponentially localised perturbations (in $L^2(\omega_0d\xi)$ with $\omega_0 \approx e^{|\xi|/2}$), which is not sufficient here as typical errors are only in $L^\infty$, such as the instability direction corresponding to the ring's mass variation $M_s Q$. To control the solution in the inner blowup zone, we thus develop here a method that extends the analysis of \cite{Sam76} (via the use of modulation and gluing techniques) to a broader class of perturbations.

\subsection{Ideas of the proof}

\paragraph{Ansatz:} In the partial mass variables (where $|\mathbb{S}^{d-1}|$ is the surface of the unit sphere in $\Rb^d$)
\begin{equation}\label{strat:def:murt}
m_u(r,t) = \frac{1}{|\mathbb{S}^{d-1}|}\int_{|x|\leq r} u(x, t)dx,  \qquad \quad r=|x|,
\end{equation}
the Keller-Segel system \eqref{3Dsys} for spherically symmetric solutions becomes
\begin{equation}\label{strat:eq:murt}
\pa_t m_u = \pa_r^2 m_u -\frac{d-1}{r}\pa_r m_u + \frac{m_u \pa_r m_u}{r^{d-1}}, \qquad r \in \Rb_+.
\end{equation}
The solutions of Theorem \ref{theo:1} correspond to solutions of the form
\be \label{strategy:mudecomp}
m_u(r,t)= M(t)Q\left(\frac{r-R(t)}{\lambda(t)} \right)+\tilde m_u(r,t), \qquad Q(\xi)=\frac{e^{\xi/2}}{1+e^{\xi/2}},
\ee
with $Q$ the \emph{traveling wave} of the viscous Burgers equation $\pa_s f=\pa_{\xi \xi }f+f\pa_\xi f$, and $\lambda=M^{-1}R^{d-1}$. Our aim is to construct a solution to \eqref{strat:eq:murt} of the form \eqref{strategy:mudecomp} with $M(t)\to M_\infty$, $R(t)\sim c_d M_\infty^{1/d}(T-t)^{1/d}$ and $\tilde m_u(t)\to 0$ as $t\uparrow T$.

\vspace{0.2cm}

\paragraph{Inner blowup zone:} We define the inner blowup variables as
\begin{equation}\label{strat:def:xis}
 \xi=\frac{r-R}{\lambda}, \qquad s=s_0+\int_0^t \frac{d\tilde t}{\lambda}, \qquad m_u(r,t)=M(t)\big[Q(\xi)+m_q(s,\xi)\big].
\end{equation}
We call the \emph{inner blowup zone} the set $\{\xi_{A,-}\leq \xi \leq \xi_{A,+}\}$ for $\xi_{A,+}=-\xi_{A,-}\gg 1$ to be fixed suitably below, and introduce $m_q^\inn=\chi^\inn m_q$ for some cut-off $\chi^\inn $ localising in the inner blowup zone. It solves (see \eqref{eq:mqxis}):
\be \label{strat:eqmq}
\pa_s m_q^\inn=\Ls_0 (m_q^\inn) +(\frac{R_\tau}{R}+\frac{1}{2}) \pa_\xi Q-\frac{M_s}{M}( Q+\xi \pa_\xi Q) +\tilde \Psi + [\pa_s -\Ls_0,\chi^\inn]m_q+ h.o.t,
\ee
where $\Ls_0 = \partial_\xi^2 - \left(1/2- Q\right)\partial_\xi + \pa_\xi Q$, $\Psi$ is the error term generated by $Q$, $[\cdot,\cdot]$ is the commutator and $ [\pa_s -\Ls_0,\chi^\inn]m_q$ are the boundary terms, and $h.o.t$ denotes higher order linear terms and nonlinear terms.

In Proposition \ref{pr:spectralL0}, we recall that the operator $\Ls_0$ is self-adjoint in $L^2(\omega_0d\xi)$ where $\omega_0(\xi)=Q^{-1} e^{\xi/2}$. It has a \emph{spectral gap} on functions such that $\int_{\mathbb R}m_q^\inn \pa_\xi Q \omega_0 d\xi=0$ resulting in exponential decay for the linear evolution:
\be \label{strat:decayinn}
\| e^{s\Ls_0} (m_q^\inn) \|_{L^2(\omega_0d\xi)}\leq e^{-\kappa' s }\|  m_q^\inn \|_{L^2(\omega_0d\xi)}.
\ee

\paragraph{Outer blowup zone:} We define the outer blowup variables
$$
\zeta=\frac{r}{R}=1+\nu \xi, \qquad \nu =\frac{R^{d-1}}{M}, \qquad \tau=\tau_0+\int_0^t \frac{M}{R^d}d\tilde t, \qquad m_\varepsilon(\tau,\zeta)=m_q(s,\xi),
$$
so that the concentrating ring is located at $\zeta=1$, and the \emph{outer blowup zone} as the set $\{\zeta<\zeta_-\}\cup \{\zeta > \zeta_+\}$ for $|\zeta_{\pm}-1|\gg \nu$ to be fixed suitably below. Then $m_\ep$ solves\footnote{Note that, comparing with \eqref{strat:eqmq}, the term corresponding to $(\frac{R_\tau}{R}+\frac{1}{2}) \pa_\xi Q$ has been incorporated in the $h.o.t.$ in \eqref{strat:eqmep} due to the decay $|\pa_\xi Q|\lesssim e^{-|\xi|/2}$.} (see \eqref{eq:mep}) the equation
\be \label{strat:eqmep}
\pa_\tau m_\ep = \As m_\ep   - \frac{M_\tau}{M} Q_\nu+\bar \Psi+h.o.t,  \qquad \textup{for}\quad  \zeta \geq \zeta_+,
\ee
where $\As = \left( \zeta^{1-d} - \zeta/2 \right) \pa_\zeta   + \nu \pa_\zeta^2$ and $Q_\nu (\zeta)=Q(\xi)$. There holds a similar equation for $\zeta\leq \zeta_-$. Equation \eqref{strat:eqmep} \emph{dampens derivatives}, as $\pa_\zeta Q_\nu \approx 0$ for $\zeta\geq \zeta_+$ and $\pa_\zeta m_\ep$ solves (see \eqref{eq:mep1})
\begin{equation}\label{strat:eq:mep1}
\pa_\tau (\pa_\zeta m_{\ep})= \As_1(\pa_\zeta m_{\ep}) +\pa_\zeta \bar \Psi+h.o.t. \qquad \textup{for}\quad  \zeta \geq \zeta_+,
\end{equation}
where $\As_1= -( (d-1)\zeta^{-d} +1/2 )+\As$ displays exponential decay
\be \label{strat:decayout}
\| e^{\tau \As_1} (\pa_\zeta m_\varepsilon^\inn )\|_{L^\infty }\leq e^{-\kappa \tau }\| \pa_\zeta m_\varepsilon^\inn \|_{L^\infty }.
\ee

\paragraph{Gluing inner and outer zones:}

\underline{Choosing $\zeta_{\pm}$.} The two time scales are such that $\tau\ll s$, and the slowest linear decay between \eqref{strat:decayinn} and \eqref{strat:decayout} is $\max(e^{-\kappa's},e^{-\kappa \tau})=e^{-\kappa \tau}$. As the outer zone with the slower a priori decay \eqref{strat:decayout} interacts with the inner zone via the boundary terms in \eqref{strat:eqmq}, we thus relax \eqref{strat:decayinn} and actually show in Lemma \ref{lemm:mqinn} that the solution to \eqref{strat:eqmq} satisfies the \emph{energy estimate}
\be \label{strat:decayinn2}
\| m_q(s)\|_{\inn}\leq K e^{-\kappa \tau},
\ee
where $\| m_q\|_\inn=-\int m_q^\inn \Ls_0 m_q^\inn \omega_0d\xi$ is a coercive functional, see Lemma \ref{lemm:coerLs0}. Since $\omega_0(\xi) \approx e^{|\xi|/2}$, after applying parabolic regularization (Lemma \ref{lem:rightboundary}), we prove that the weighted $L^2$ bound \eqref{strat:decayinn2} implies the pointwise bound for the derivative $|\pa_\xi m_q |\lesssim K e^{-|\xi|/4}e^{-\kappa \tau}$, so that $|\pa_\zeta m_\varepsilon| \lesssim K\nu^{-1}e^{-|\xi|/4}e^{-\kappa \tau}$ as $\pa_\xi=\nu \pa_\zeta$. This later estimate matches with \eqref{strat:decayout} precisely for $\nu^{-1}e^{-|\xi|/4}=1$ corresponding to the choice
$$
\zeta_{\pm}= 1\pm 4\nu |\log \nu|.
$$
\underline{Choosing $\xi_{A,\pm}$.} The transport field $ \left( \zeta^{1-d} - \zeta/2 \right) \pa_\zeta$ in \eqref{strat:eqmep} pushes \emph{from the outer blowup zone toward the inner zone}. Thus, the farther from the inner zone, the smaller the effects of boundary terms should be. This is made rigorous in Lemma \ref{lemm:mep_mid} where we prove
\be \label{strat:decayout2}
|\pa_\zeta m_\ep | \leq K^{5/4} \phi_1+ \phi_2 \qquad \zeta\geq \zeta_+, \qquad \textup{with}\;\; \phi_1 =  e^{-\kappa \tau}e^{\frac 38 \frac{\zeta-\zeta_+}{\nu}} \mbox{ and } \phi_2=e^{-\kappa \tau}\zeta^{d-1}
\ee
by \emph{parabolic comparison principle}, and a similar estimate for $\zeta\leq \zeta_-$ holds. The supersolution $\phi_1$ takes care of the boundary condition $\pa_\zeta m_\ep (\zeta_+)$ imposed by the inner zone including the viscosity effect. By choosing
$$
\xi_{A,\pm}= \pm (4|\log \nu|+A),
$$
where $A\gg 1 $ is such that $e^{3A/10}\leq K\leq e^{A/2}$, we show in the proof of Lemma \ref{lemm:mqinn} that \eqref{strat:decayout2} implies
$$
\| [\pa_s -\Ls_0,\chi^\inn]m_q\|_{L^2(\omega_0d\xi)}\lesssim (K^{1/4}e^{-A/8}) K e^{-\kappa \tau}\ll K e^{-\kappa \tau},
$$
for the boundary terms in \eqref{strat:eqmq}, which is compatible with \eqref{strat:decayinn2}.

Note that the inner and outer blowup zones \emph{overlap} in $\{\xi_{A,-}\leq \xi \leq \xi_-\}\cup \{\xi_{+}\leq \xi \leq \xi_{A,+}\}$, where we obtain a delay-type estimate for the associated parabolic tranport equations.

\vspace{0.2cm}

\paragraph{Nonlinear analysis:} To handle nonlinear effects, the solution is controlled in a bootstrap regime, see Definition \ref{def:bootstrap}. The parameters $R$ and $M$, that are related to instability directions around the approximate solution, are determined dynamically from \eqref{strat:eqmq} by requiring the orthogonality $\int_{\mathbb R}  m_q^\inn \pa_\xi Q\omega_0 d\xi=0$ and the cancellation $\int_{\xi_{A,+}}^{\xi_{A,+}+1} m_q^\inn d\xi=0 $ respectively. This yields the dynamical system \eqref{eq:ModR}-\eqref{eq:ModM} driving the blowup. The other nonlinear terms are treated perturbatively using Sobolev-type estimates.

\subsection{Aknowledgments}

The work of C. Collot was funded by CY Initiative of Excellence (Grant "Investissements d'Avenir" ANR-16-IDEX-0008). T. E. Ghoul, N. Masmoudi and V.T. Nguyen are supported by Tamkeen under the NYU Abu Dhabi Research Institute grant of the center SITE. N. Masmoudi is supported by NSF grant DMS-1716466.

\section{Formulation of the problem} \label{sec:3}

\noindent We will work in the partial mass variables (with $|\mathbb{S}^{d-1}|$ the surface of the unit sphere in $\Rb^d$)
\begin{equation}\label{def:murt}
m_u(r,t) = \frac{1}{|\mathbb{S}^{d-1}|}\int_{|x|\leq r} u(x, t)dx,  \qquad \quad r=|x|,
\end{equation}
in which the Keller-Segel system \eqref{3Dsys} for spherically symmetric solutions becomes:
\begin{equation}\label{eq:murt}
\pa_t m_u = \pa_r^2 m_u -\frac{d-1}{r}\pa_r m_u + \frac{m_u \pa_r m_u}{r^{d-1}}, \qquad r \in \Rb_+.
\end{equation}

\subsection{Renormalised variables} 
\subsubsection{Hyperbolic inviscid variables} These variables describe the solution away from the ring $|r-R(t)|\gg \lambda(t)$, where nonlinear transport is dominant and viscosity effects are negligible. For $R(t)$ and $M(t)$ two positive $\mathcal{C}^1$ functions we define the parameters
\begin{equation}\label{eq:relRMnu}
\nu = \frac{R^{d-2}}{M}, \qquad \lambda=R\nu,
\end{equation}
so that
$$
\nu_\tau=-\frac{d-2}{2}\nu+(d-2)\left(\frac{R_\tau}{R}+\frac 12 \right)\nu-\frac{M_\tau}{M}\nu,
$$
and the renormalization
\begin{equation} \label{def:mw}
m(r,t) = M(t)m_w(\zeta, \tau), \qquad \zeta=\frac{r}{R}, \qquad \tau=\tau_0+\int_0^t \frac{M(\tilde t)}{R^d(\tilde t)}d\tilde t.
\end{equation}
The new unknown $m_w(\zeta,\tau)$ satisfies for $\zeta> 0$ and $\tau \geq \tau_0$,
\begin{align}
\partial_\tau m_w =  \left(\frac{m_w}{\zeta^{d-1}}-\frac{1}{2}\zeta\right) \pa_\zeta m_w +   \nu \zeta^{d-1}\pa_\zeta \left( \frac{\pa_\zeta m_w}{\zeta^{d-1}} \right) + \left(\frac{R_\tau}{R}+\frac 12 \right)\zeta \pa_\zeta m_w -\frac{M_\tau}{M}m_w. \label{eq:phisfrsft}
\end{align}

\begin{remark} \label{re:shock}

We shall prove that $\nu$ goes to zero as $\tau \to \infty$. Then, we notice that with the special choice
$$m_w = \mathbf{1}_{\{\zeta \geq 1\}}, \quad \frac{R_\tau}{R} = -\frac{1}{2}, \quad M_\tau = 0,$$
the inviscid equation \eqref{eq:phisfrsft}, i.e. with $\nu = 0$, is solved both sides of the discontinuous point $\zeta = 1$. The Rankine-Hugoniot condition
\be \label{id:rankinehugoniot}
\frac{1}{2}\left[\lim_{\zeta \to 1^+} \left( \frac{\mathbf{1}_{\{\zeta \geq 1\}}}{\zeta^{d-1}} - \frac{1}{2}\zeta   \right) + \lim_{\zeta \to 1^-} \left( \frac{\mathbf{1}_{\{\zeta \geq 1\}}}{\zeta^{d-1}} - \frac{1}{2}\zeta   \right) \right] = 0,
\ee
asserts that the discontinuous point $\zeta = 1$ is steady so that $\mathbf{1}_{\{\zeta \geq 1\}}$ defines a stationary solution for the limiting inviscid equation. The function ${\bf 1}_{\{\zeta \geq 1\}}$ will be the blow-up profile in the hyperbolic inviscid variables \eqref{def:mw}.

\end{remark}

\subsubsection{Blowup variables inside the ring} To have a better description near the shock location $\zeta = 1$ (the appearance of a shock being explained in Remark \eqref{re:shock}) we change variables
\begin{equation}\label{def:xis}
m_w(\zeta,\tau)=m_v(\xi,s),\qquad \xi=\frac{\zeta-1}{\nu}=\frac{r-R}{R\nu}, \qquad s=s_0+\int_{\tau_0}^\tau \frac{d\tau}{\nu}.
\end{equation}
Then $m_v$ solves the following equation for $\xi>-\nicefrac{1}{\nu}$ and $s\geq s_0$:
\begin{align}
\pa_s m_v &=\pa_\xi^2 m_v +m_v\pa_\xi m_v-\frac{1}{2}\pa_\xi m_v+\left(\frac{R_\tau}{R}+\frac{1}{2} \right)\pa_\xi m_v-\frac{M_s}{M}m_v\nonumber \\
&\qquad +\left(\frac{1}{(1+\xi\nu)^{d-1}}-1 \right)m_v \pa_\xi m_v -\nu \frac{d-1}{1+\nu\xi}\pa_\xi m_v. \label{eq:vxis} \\
\nonumber & \qquad \quad + \left( (d-1)\nu \left(\frac{R_\tau}{R}+\frac 12\right)-\frac{d-1}{2}\nu-\frac{M_s}{M} \right) \xi \pa_\xi m_v
\end{align}
As we expect $R_\tau\sim -R/2$, $M_s\approx 0$ and $\nu \to 0$, we introduce the blowup profile $Q$ near the shock that cancels out the leading part in \eqref{eq:vxis}, namely $Q$ solves the ODE
\begin{equation}\label{eq:Q}
\pa_\xi^2 Q - \frac 12 \pa_\xi Q + Q\pa_\xi Q = 0, \quad \lim_{y \to -\infty} Q(\xi) = 0,
\end{equation}
whose exact solution is given by 
\begin{equation}\label{def:Qxi}
Q(\xi) =  e^{\frac \xi2}\big(1 + e^{\frac \xi2}\big)^{-1}, \quad \pa_\xi Q(\xi) =\frac{1}{8}\cosh^{-2}(\frac \xi 4).
\end{equation}

\begin{remark} \label{re:burgers}

Keeping only the leading order terms in \eqref{eq:vxis} gives the Burgers equation $\pa_s f=\pa_\xi^2 f +f\pa_\xi f-\frac{1}{2}\pa_\xi f$, for which $Q$ is a traveling wave, since $f(\tau,\xi)=Q(\xi+\tau/2)
$ is an exact solution. It travels at speed $-1/2$ which equals the speed of the shock determined from the Rankine-Hugoniot condition \eqref{id:rankinehugoniot}.

\end{remark}

\subsection{Linearized problems}
\subsubsection{The profile} For a fixed $0 < \zeta_0 \ll 1$, we introduce $\bar \chi$ a smooth nonnegative cut-off function with
\begin{equation}\label{def:barchi}
\bar \chi(\zeta) = \left\{ \begin{array}{ll}0 & \textup{if}\;\; \zeta \in [0, \zeta_0],\\
1 & \textup{if}\;\; \zeta \in [2\zeta_0, \infty).
\end{array} \right.
\end{equation} 
We introduce the notation for the rescaled and localised profiles:
\begin{equation}\label{def:Qbar}
 Q_\nu(\zeta) = Q(\xi), \quad \quad \bar Q_\nu(\zeta)= Q_\nu(\zeta) \bar \chi(\zeta) \quad \textup{and} \quad \bar Q(\xi) = \bar Q_\nu(\zeta).
\end{equation} 
The introduction of the localised profile $\bar Q_\nu$ is technical, to deal with the singular nonlinear term at the origin. By the definition of $\bar \chi$, we note that 
\begin{equation*}
\bar Q(\xi) = 0 \quad \textup{for}\quad \xi \in \left[-\frac{1}{\nu}, -\frac{(1 - \zeta_0)}{\nu}\right], \quad \bar Q(\xi) = Q(\xi) \quad \textup{for} \quad \xi \geq -\frac{(1 - 2\zeta_ 0)}{\nu}.
\end{equation*}

 \subsubsection{Linearized equation in the partial mass setting} We introduce the decomposition in hyperbolic inviscid variables \eqref{def:mw}
\begin{equation}\label{def:mepmq}
m_w(\zeta) =   \bar Q_\nu(\zeta)+m_\ep(\zeta, \tau).
\end{equation}
The perturbation $m_\ep$ then solves the equation for $\zeta > 0$ and $\tau \geq \tau_0$:
\begin{align}
\pa_\tau m_\ep = \frac{\pa_\zeta(\bar Q_\nu m_\ep) }{\zeta^{d-1}} - \frac{1}{2}\zeta \pa_\zeta m_\ep &+\nu \left( \pa_\zeta^2 m_\ep - \frac{d-1}{\zeta} \pa_\zeta m_\ep\right) + \frac{m_\ep \pa_\zeta m_\ep}{\zeta^{d-1}} \nonumber\\
& + \left(\frac{R_\tau}{R} + \frac{1}{2} \right)\zeta \pa_\zeta m_\ep - \frac{M_\tau}{M}  m_\ep + m_E, \label{eq:mep0}
\end{align}
where the generated error is
\begin{align}
m_E = -\pa_\tau \bar Q_\nu + \frac{\bar Q_\nu \pa_\zeta \bar Q(\zeta)}{\zeta^{d-1}} - \frac{1}{2}\zeta \pa_\zeta \bar Q_\nu &+ \nu \left(\pa_\zeta^2 \bar Q_\nu - \frac{d-1}{\zeta}\pa_\zeta \bar Q_\nu \right) \nonumber\\
& +\left(\frac{R_\tau}{R} + \frac{1}{2} \right) \zeta \pa_\zeta \bar Q_\nu - \frac{M_\tau}{M}\bar Q_\nu. \label{def:mE0}
\end{align}

In the ring, in the blowup variables \eqref{def:xis}, we introduce the decomposition
\begin{equation} \label{def:mq}
m_q(\xi,s) = m_v(\xi,s) - \bar Q(\xi),
\end{equation}
that leads to the following linearized equation for $\xi > -\nicefrac{1}{\nu}$ and $s\geq s_0$, 
\begin{equation}\label{eq:mqxis}
\pa_s m_q=\Ls_0 (m_q) + L(m_q) + \frac{m_q \partial_\xi m_q}{(1 +\nu \xi )^{d-1}} + \Psi,
\end{equation}
subject to the boundary condition\footnote{Note that this boundary condition is propagated with time since we consider solutions $u$ to \eqref{3Dsys} that are in $L^\infty(\mathbb R^d)$, so that $m_u(r)=O(r^d)$ as $r\to 0$ using \eqref{def:murt}.} 
$$m_q(s,-\nicefrac{1}{\nu})=0.$$
Above, the elliptic linearized operator is defined as
\begin{equation}\label{def:Lop}
\Ls_0 = \partial_\xi^2 - \left(\frac{1}{2} - Q\right)\partial_\xi + Q',
\end{equation}
the lower order linear term is
\begin{align}
\nonumber L(m_q)&= -(d-1)\nu \left(\frac 12 \xi +\frac{1}{1+\nu \xi} \right) \pa_\xi m_q + \left(\frac{R_\tau}{R}+\frac{1}{2} \right)\left(1+(d-1)\nu \xi \right)\pa_\xi m_q-\frac{M_s}{M}\left(m_q+\xi\pa_\xi m_q\right)\\
\nonumber &\qquad +\left(\frac{1}{(1+\xi\nu)^{d-1}}-1 \right)\bar Q \pa_\xi m_q+\left(\frac{1}{(1+\xi\nu)^{d-1}}-1 \right)m_q\pa_\xi \bar Q\\
& \qquad + \pa_\xi (\bar Q - Q) m_q + (\bar Q - Q)\pa_\xi m_q,
\label{def:Lep}
\end{align}
and the generated error is given by
\begin{align}
\Psi(\xi,s) &=  \left(\frac{R_\tau}{R}+\frac{1}{2} \right)\left(1+(d-1)\nu \xi \right)\pa_\xi \bar Q-\frac{M_s}{M}\left(\bar Q+\xi \pa_\xi \bar Q\right)  -(d-1)\nu \left(\frac 12 \xi +\frac{1}{1+\nu \xi} \right) \pa_\xi \bar Q \nonumber\\
&\qquad +\left(\frac{1}{(1+\xi\nu)^{d-1}}-1 \right)\bar Q \pa_\xi \bar Q -\nu \frac{d-1}{1+\nu\xi}\pa_\xi \bar Q \nonumber\\
& \qquad \quad + \nu \pa_\zeta \bar \chi \Big( 2\pa_\xi Q - \frac{1}{2}Q + Q^2 \bar \chi\Big) + \nu^2\pa_{\zeta}^2 \bar \chi Q + Q \pa_\xi Q \bar \chi \big(\bar \chi - 1\big)-\nu_s \xi \pa_\zeta \bar \chi Q. \label{def:Psi}
\end{align}
\subsubsection{Evaluation of the parameters} The parameter functions $R$ and $M$ are determined via the  "orthogonality" conditions:
\begin{equation} \label{orthogonality2}
\int_{-1/\nu}^\infty \chi_{_A}(\xi) m_q(s,\xi)d\xi=2 \int_{-1/\nu}^\infty \chi_{_A}(\xi) m_q(s,\xi) \pa_\xi Q(\xi) \omega_0(\xi)d\xi=0,
\end{equation}
and 
\begin{equation} \label{orthCon2}
\int_{-1/\nu}^\infty m_q(s,\xi) \chi_{_{1,\; \xi_{A,+}} }(\xi) d\xi =0,
\end{equation}
where $\xi_{A,+}$ is defined in \eqref{def:xipm}, and for any positive constants $A$ and $a$, $\chi_{_A}$ and $\chi_{_{A, a}}$ are cut-off functions defined by 
\begin{equation}\label{def:chiA}
 \chi_{_{A, a}}(\xi) = \chi_0\Big( \frac{x - a}{A}\Big), \quad \chi_{_A}(\xi) = \chi_{_{A,0}}(\xi), 
\end{equation}
where $\chi_0$ is smooth and nonnegative, and satisfies
\begin{equation*}
\chi_0 \in \Cc^\infty(\Rb), \quad  \chi_0(x) =  \left\{ \begin{array}{l} 0 \; \textup{if} \; |x| \geq 2,\\
1 \; \textup{if}\; |x| \leq 1,\\
\end{array}\right.
\end{equation*}
and $\omega_0$ is the weight function
\begin{equation} \label{def:omega}
\omega_0(\xi) = \left(e^{\frac{\xi}{4}}+e^{-\frac{\xi}{4}} \right)^2=\frac{1}{2 \pa_\xi Q(\xi)} .
\end{equation}
\begin{remark} The orthogonality condition \eqref{orthogonality2} ensures a coercivity estimate for the linearized operator $\Ls_0$ as stated in Lemma \ref{lemm:coerLs0}.  By the mean value theorem the second condition \eqref{orthCon2} implies that there exists a point $\xi_* \in \big(\xi_{A,+}-2, \xi_{A,+}+2\big)$ where $\xi_{A,+}$ is defined in \eqref{def:xipm} such that $m_q(\xi_*,s) = 0$. This allows us to write
\begin{equation}\label{est:mqbydev}
m_q(\xi, s) = \int_{\xi_*}^\xi \partial_\xi m_q(\xi, s) d\xi, \quad \textup{hence,}\quad  |m_q(\xi, s)| \leq |\xi - \xi^*| \| \pa_\xi m_q( s)\|_{L^\infty (\xi_*,\xi)}. 
\end{equation}
\end{remark}
%

\subsection{The linearized operator around the Burgers traveling wave}

The linearized operator $\Ls_0$ appears in the study of stability of traveling wave solutions to the viscous Burger equation. Its properties are thus well-known. We define for $k\in \mathbb N$ the weighted Sobolev space $H^k_{\omega_0}$ associated to the norm
$$
\| m\|_{H^k_{\omega_0}}^2=\sum_{j=0}^k \int_{\mathbb R} (\pa_\xi^j u)^2\omega_0(\xi)d\xi .
$$

\begin{proposition} \label{pr:spectralL0}

The operator $\Ls_0$, with domain $H^2_{\omega_0}$, is self-adjoint on $L^2_{\omega_0}(\mathbb{R})$. Its spectrum consists of an isolated eigenvalue which is $0$ associated to the eigenfunction $\pa_\xi Q$, and of the interval $(-\infty,-1/16]$.

\end{proposition}

\begin{proof} Proposition \ref{pr:spectralL0} is obtained in \cite{Sam76}, but one argument in the proof contains an error that can be corrected. We thus give a proof here for sake of completeness and mention where we correct the error using an identity of \cite{Njjam85} related to the Cole-Hopf transformation.

Equation \eqref{eq:Q} is invariant by space translation, hence the function $\pa_\xi Q$ satisfies $\Ls_0 \pa_\xi Q=0$. It belongs to $L^2_{\omega_0}$ since $|\pa_\xi Q(\xi)|\lesssim e^{-|\xi|/2}$ and $\omega_0\approx e^{|\xi|/2}$. Standard ODE arguments show that any other solution to $\Ls_0 m=0$ that is non collinear to $\pa_\xi Q$ has nonzero finite limits as $\xi \rightarrow \pm \infty$, preventing them to belong to $L^2_{\omega_0}$. Hence $\pa_\xi Q$ spans the kernel of $\Ls_0$ in $L^2_{\omega_0}$.

The eigenfunction $\pa_\xi Q$ associated to $0$ is positive on $\mathbb R$. A Sturm-Liouville argument (see \cite{Sam76}) then implies that $\Ls_0$ has no positive eigenvalues.

To study the rest of the spectrum, it is observed in \cite{Sam76} that $\Ls_0$ can be written under the following conjugated form as \footnote{Any operator of the form $\Lc = \partial_y^2 - 2b \partial_y + c$ can be written as $\Lc = e^B \mathcal{M} e^{-B}$, where $B(y) = \int_0^y b(\xi) d\xi$ and $\mathcal{M} = \partial_y^2 + [b' - b^2 + c]$. A similar formulation holds for the  higher dimensional case, namely that $\Lc = \Delta - 2b.\nabla +c$ can be written as $\Lc = e^{B} \mathcal{M} e^{-B}$ with $\nabla B = b$ and $\mathcal{M} = \Delta + [\Delta B - |\nabla B|^2 + c]$.}
\begin{equation}
\Ls_0 = e^{B_0} \Ms_0 e^{-B_0} \quad \textup{with} \quad B_0(\xi) = \int_0^\xi b_0(\tilde \xi) d\tilde \xi, \quad  b_0 = \frac{1}{4} - \frac{Q}{2}.
\end{equation}
A mistake was made in \cite{Sam76} in the computation of $\Ms_0$, and the correct operator is given by
\begin{equation}
\Ms_0 = \partial^2_\xi  +\left[ \frac{1}{2} \frac{Q}{(1 + e^{\nicefrac{\xi}{2}})}- \frac{1}{16} \right].
\end{equation}
The operator $\Ms_0$ on $L^2(\mathbb R)$ (with domain $H^2(\mathbb R)$) has continuous spectrum in the interval $-\infty < \lambda \leq -\frac{1}{16}$, since it is a compact perturbation of $\pa_{\xi}^2-\frac{1}{16}$. Hence, we deduce that $\Ls_0$ has the same continuous spectrum $(-\infty,-1/16]$. It remains to show that there are no eigenvalues in $(-1/16,0)$. We give a different argument from that in \cite{Sam76} which relied on the aforementioned erroneous identity of $ \Ms_0$. Assume by contradiction that there exists $c\in (0,1/16)$ and $\psi \in H^2_{\omega_0}$ such that $\Ls_0 \psi=-c\psi$. Since $\Ls_0$ is self-adjoint in $L^2_{\omega_0}$ and $2\pa_\xi Q=\omega_0^{-1}$ is another eigenfunction
\be \label{spectral:zeromean}
\int_{\mathbb R} \psi=2\int_{\mathbb R} \psi \pa_\xi Q \omega_0 =0.
\ee
We claim moreover that for some $C>0$, there holds
\be \label{spectral:claim}
|\psi(\xi)|\leq Ce^{-\mu |\xi|}, \qquad \mu=\frac 14 (1+\sqrt{1-16 c}),
\ee
whose proof is done shortly after. Letting $ \phi(\xi)=\int_{0}^\xi \psi(\eta)d\eta$, by \cite[Theorem 2]{Njjam85} we have
$$
(e^{s \Ls_0} \psi)(\xi)=\int_{\mathbb R} \pa_\xi \tilde \Gamma(\xi,s,\eta) \phi(\eta)d\eta, \qquad \tilde \Gamma (\xi,s,\eta)= e^{-\frac{s}{16}} \frac{e^{-\frac{(\xi-\eta)^2}{4s}}}{\sqrt{4\pi s}}e^{\frac{1}{2}\int_{\eta}^\xi (\frac 12 -Q(\zeta))d\zeta}.
$$
Let $\xi_0\in \mathbb R$ such that $\psi(\xi_0)\neq 0$, we fix $\psi(\xi_0)=1$ without loss of generality. Then
\be \label{spectral:identity}
e^{-c s}=\int_{\mathbb R} \pa_\xi \tilde \Gamma(\xi,s,\eta) \phi(\eta)d\eta.
\ee
On the other hand,  we estimate using \eqref{def:Qxi},  $|\int_\eta^{\xi_0}(\frac 12-Q)|\leq C(\xi_0)+\frac 12 |\eta|$, from which and $|\zeta e^{-\zeta}|\lesssim 1$ we obtain for $s\geq 1$ that $|\pa_\xi \tilde \Gamma (\xi_0,s,\eta)|\lesssim e^{-\frac{s}{16}+\frac{|\eta|}{4}}$. Combining this, \eqref{spectral:zeromean} and \eqref{spectral:claim} yields for $s\geq 1$,
$$
\left| \int_{\mathbb R} \pa_\xi \tilde \Gamma(\xi,s,\eta) \phi(\eta)d\eta \right|\lesssim e^{-\frac{1}{16}s}\int_{\mathbb R} e^{(\frac 14-\mu )|\eta|}d\eta \lesssim e^{-\frac{1}{16}s}.
$$
This contradicts \eqref{spectral:identity} for $s$ large, hence, $\Ls_0$ has no eigenvalues in $(-1/16,0)$.

It remains to prove \eqref{spectral:claim}. Using \eqref{def:Qxi} we write
\be \label{spectral:asympt}
((\Ls_0+c)\psi)(\xi)=(\Ls_\infty \psi)(\xi)+O(e^{-\frac{|\xi|}{2}})\pa_\xi \psi(\xi)+O(e^{-\frac{|\xi|}{2}})\psi(\xi) \quad \mbox{as }\xi\to \infty,
\ee
where $\Ls_\infty=\pa_\xi^2+\frac 12 \pa_\xi +c$. The solutions of $\Ls_\infty f=0$ are $f_\pm (\xi)=e^{\lambda_\pm \xi}$ with $\lambda_\pm =\frac 14 (-1\pm \sqrt{1-16 c})$. By standard ODE arguments, as $(\Ls_0+c)\psi=0$, \eqref{spectral:asympt} implies that there exists $\iota \in \{\pm 1\}$ and $c_\infty \neq 0$ such that $\psi(\xi)\sim c_\infty e^{\lambda_\pm \xi}$ as $\xi \to \infty$. As $\psi\in L^2_{\omega_0}$ and $\omega_0 \approx e^{|\xi|/2}$ necessarily $\iota=+1$ so $|\psi(\xi)|\lesssim e^{-\mu \xi}$ for $\xi \geq 0$. The proof of $|\psi(\xi)|\lesssim e^{\mu \xi}$ for $\xi \leq 0$ is similar, yielding \eqref{spectral:claim}.
\end{proof}

For any $m \in H^1_{\omega_0}$,  one obtains by integration by parts,
\be \label{id:formulamLsm}
\int_{\mathbb R} m \Ls_0 m \omega_0d\xi=- \int_{\mathbb R} |\pa_\xi m|^2 \omega_0d\xi+\int_{\mathbb R}m^2 \pa_\xi Q \omega_0d\xi.
\ee
The above bilinear form is coercive outside the kernel of $\Ls_0$ as shown in the following lemma.

\begin{lemma}[Coercivity of $\Ls_0$] \label{lemm:coerLs0} There exist $\delta > 0$ such that for all $m \in H^1_{\omega_0}$ we have:
\begin{equation} \label{bd:coercivityquadraticform}
\langle \Ls_0 m, m\rangle_{L^2_{\omega_0}} \leq -\delta \| m\|_{H^1_{\omega_0}}^2+\langle m, \pa_\xi Q\rangle_{L^2_{\omega_0}}^2. 
\end{equation}
\end{lemma}
\begin{proof}
From Proposition \ref{pr:spectralL0} and the spectral Theorem, for any $m\in H^1_{\omega_0}$ such that $\langle m ,\pa_\xi Q \rangle_{L^2_{\omega_0}}=0$ there holds $\langle \Ls_0 m, m\rangle_{L^2_{\omega_0}} \leq - 1 / 16 \ \| m \|_{L^2_{\omega_0}}^2$. Hence, for $m\in H^1_{\omega_0}$, we have
$$
\langle \Ls_0 m, m\rangle_{L^2_{\omega_0}} \leq - \frac{1}{16}\| m \|_{L^2_{\omega_0}}^2+\frac{1}{16} \frac{\langle m , \pa_\xi Q \rangle_{L^2_{\omega_0}}^2}{\| \pa_\xi Q \|_{L^2_{\omega_0}}^2}.
$$
We use the formula \eqref{id:formulamLsm}, the above inequality and $|\pa_\xi Q|\leq \frac12$ from \eqref{def:Qxi} to write for $\delta \in (0, 1/9)$:
\begin{align*}
\langle \Ls_0 m, m\rangle_{L^2_{\omega_0}} & = -\delta \int_{\mathbb R} |\pa_\xi m|^2 \omega_0 +\delta\int_{\mathbb R}m^2 \pa_\xi Q \omega_0 +(1 - \delta)\langle \Ls_0 m, m\rangle_{L^2_{\omega_0}} \\
&\leq -\delta\int_{\mathbb R} |\pa_\xi m|^2 \omega_0 +\frac{\delta}{2}\int_{\mathbb R}m^2 \omega_0 - \frac{1-\delta}{16}\| m \|_{L^2_{\omega_0}}^2+\frac{(1 - \delta)}{16} \frac{\langle m , \pa_\xi Q \rangle_{L^2_{\omega_0}}^2}{\| \pa_\xi Q \|_{L^2_{\omega_0}}^2}\\
&\leq -\delta \| m\|_{H^1_{\omega_0}}^2+ \frac{\langle m , \pa_\xi Q \rangle_{L^2_{\omega_0}}^2}{\| \pa_\xi Q \|_{L^2_{\omega_0}}^2},
\end{align*}
which is the desired estimate\eqref{bd:coercivityquadraticform}.
\end{proof}

\subsection{Bootstrap regime} \label{sec:bootstrap}

We introduce for a constant $A > 0$ to be fixed later on:
\begin{align} \label{def:zetapm}
& \zeta_{\pm}=1\pm 4\nu |\log \nu|, \qquad  \zeta_{A,\pm}=1\pm \nu (4|\log \nu|+A),\\
 \label{def:xipm}
&\xi_{\pm}=\pm 4 |\log \nu|, \ \ \quad \qquad  \xi_{A,\pm}=\pm (4|\log \nu|+A),
\end{align}
and will refer to the zone $\zeta_{A,-} \leq \zeta \leq \zeta_{A,+}$ as the inner zone, and to the zone $\{ 0<\zeta\leq \zeta_-\} \cup \{ \zeta\geq \zeta_+\}$ as the outer zone. Note that these two zones overlap on $\{\zeta_{A,-}\leq \zeta \leq \zeta_-\}\cup \{ \zeta_+\leq \zeta \leq \zeta_{A,+}\}$.

Let $\chi_1$ be a smooth nonnegative cut-off, with $\chi_1(\xi)=1$ for $\xi\leq 0$ and $\chi_1(\xi)=0$ for $\xi\geq 1$. We define
$$
 \chi^\inn (s,\xi)=\chi_1(\xi-\xi_{A,+})\chi_1(\xi_{A,-}-\xi).
$$
Note that $\textup{supp}(\pa_\xi \chi^\inn)\subset [\xi_{A,-}-1,\xi_{A,-}]\cup [\xi_{A,+},\xi_{A,+}+1]$. We introduce
\begin{equation}\label{def:chi_inn}
 m_q^\inn (s,\xi) = \chi^\inn (s,\xi) m_q (s,\xi).
\end{equation}
The two main norms to control the remainder in our analysis are $\|m_q^\inn\|_{L^2_{\omega_0}}$ and a weighted $L^\infty$ bound for $\pa_\zeta m_\varepsilon$ for $\zeta\leq \zeta_{-}$ and $\zeta \geq \zeta_+$, from which we are able to derive the leading dynamical system driving the law of blowup solutions as described in Theorem \ref{theo:1}. The influence of the exterior zone on the interior one is measured by the  quantity
\begin{align} \label{def:boundarynorm}
\| m_\varepsilon \|_{\bd}& = \| m_\varepsilon \|_{L^\infty( [\zeta_{A,-}-2\nu,\zeta_{A,-}+2\nu]\cup [\zeta_{A,+}-2\nu,\zeta_{A,+}+2\nu ])}\\
\nonumber &\qquad +\nu \| \pa_\zeta m_\varepsilon \|_{L^\infty( [\zeta_{A,-}-2\nu,\zeta_{A,-}+2\nu]\cup [\zeta_{A,+}-2\nu,\zeta_{A,+}+2\nu ])}.
\end{align}
Since the norm $\|m_q^\inn\|_{L^2_{\omega_0}}$ itself is not enough to close nonlinear estimates, we introduce the adapted higher order regularity norm
\begin{equation}\label{def:inn}
\|m_q\|^2_\inn = -\int_{-\nicefrac{1}{\nu}}^\infty  m_q^\inn\Ls_0 m_q^\inn \omega_0 d\xi.
\end{equation}
Thanks to the coercivity of $\Ls_0$ given by \eqref{coer:id:coercivity2} and the orthogonality condition \eqref{orthogonality2}, we have the equivalence 
\begin{equation}\label{eq:normequiv}
\|m_q\|_\inn \sim \| m_q^\inn\|_{H^1_{\omega_0}}. 
\end{equation}
For a fixed small constant $0 < \eta \ll 1$, we introduce $ \hat \chi_{\eta}$ a smooth cut-off function defined as
\begin{equation}\label{def:hatchi_eta}
 \hat \chi_{\eta}(\zeta) = \left\{ \begin{array}{ll} 1 \;\; \textup{for}\;\; |\zeta - 1| \leq \eta,\\
0 \;\; \textup{for}\;\; |\zeta - 1| \geq 2 \eta.
\end{array}  \right.
\end{equation}
We define the following bootstrap estimates.

\begin{definition}[Bootstrap regime] \label{def:bootstrap} For $A,K ,\kappa,\eta ,M_0>0$ and $\tau > 0$, we define $\Sc(\tau)=\Sc [A,K,\kappa,\eta,M_0](\tau)$ as the set of all functions $m_u\in C^1 ((0, \infty),\mathbb R)$ for which there exist $M(\tau),R(\tau)>0$ with
\be
\frac{e^{-\frac \tau 2}}{4}< R(\tau)  < 4e^{-\frac \tau 2}, \qquad \frac{M_0}{4}< M(\tau)< 4 M_0 \label{est:MR}
\ee
such that $m_\varepsilon$ defined as in the decomposition \eqref{def:mepmq} satisfies
\begin{align}
\big| \pa_\zeta m_\ep(\zeta, \tau)\big| &< e^{-\kappa \tau} \left(K^{\frac{5}{4}} e^{-\frac 38 \frac{\zeta-\zeta_+}{\nu} }  \hat \chi_{\eta} + \zeta^{d-1} \right), \qquad \textup{for}\; \zeta \geq \zeta_+, \label{est:mepLinf_out} \\
\big| \pa_\zeta m_\ep(\zeta, \tau)\big| &< e^{-\kappa \tau} \left(K^\frac{5}{4} e^{-\frac{3}{8} \frac{ \zeta_--\zeta}{\nu} } \hat \chi_{\eta} +\nu \zeta^{d-1}  \right)  \qquad \textup{for}\; 0 < \zeta \leq \zeta_-. \label{est:mepLinf_out2}
\end{align}
and $m_q$ defined as in the decomposition \eqref{def:mq} satisfies the orthogonality conditions \eqref{orthogonality2} and \eqref{orthCon2} and
\be \label{est:mqL2ome}
\|m_q(\tau) \|_\inn < K e^{-\kappa \tau}.
\ee

\end{definition}

\begin{remark}
The specific constant $\frac{3}{8}$ is just for a sake of simplification and can be any real number in the interval  $\big(\frac 14,\frac 12\big)$. The two constants $A$ and $K$ will chosen such that $e^{3A/10}\leq K \leq e^{A/2}$ to ensure certain estimates. The points $\zeta= 1\pm 4\nu |\log \nu |$ are chosen so that linear estimates in the inner and exterior zones are compatible at these points.

\end{remark}

\noindent We claim the following proposition which is central for our analysis.
\begin{proposition}[Existence of solutions to \eqref{eq:mqxis} trapped in $S(\tau)$]  \label{prop:trap_mq} There exist constants $K,A \gg 1$, $0 < \kappa,\eta \ll 1$ and a function $\bar M_0\mapsto \tau_0^*(\bar M_0)$, such that for any $\bar M_0>0$, for any $M_0\geq \bar M_0$ and $\tau_0\geq \tau_0^*$, if initially
\be \label{bootstrap:choiceR0}
R(\tau_0)=e^{-\frac{\tau_0}{2}}, \qquad M(\tau_0)=M_0,
\ee
and $m_\varepsilon(0)$ satisfies
\begin{align}
& \label{bd:bootstrap:initial} m_u (0) \in \Sc [A,1,\kappa,\eta,M_0](\tau_0), \\
& \label{bd:bootstrap:initialregularity}
|\pa_\zeta m_\varepsilon (\tau_0) | < \frac 12 e^{-\kappa \tau_0}\zeta^{d-1} \qquad \mbox{for} \quad \zeta \geq \zeta_+(0),\\
& |\pa_\zeta m_\varepsilon (\tau_0) | < \frac 12 \nu_0 e^{-\kappa \tau_0}\zeta^{d-1} \qquad \mbox{for} \quad 0 \leq \zeta \leq \zeta_-(0),
\end{align}
where $\zeta_\pm(0)=1\pm 4 \nu_0 |\log \nu_0| $ with $\nu_0=R^{d-2}(\tau_0)M^{-1}(\tau_0)$. Then, the solution to \eqref{eq:mep0} with the initial datum $m_\varepsilon(0)$ exists for all $\tau \geq \tau_0$ and belongs to $\Sc [A,K,\kappa,\eta,M_0](\tau)$ for all $\tau \in [\tau_0, +\infty)$.
\end{proposition}

 We postpone the proof of Proposition \ref{prop:trap_mq}  to Section \ref{sec:4.4} as it is a consequence of improved estimates obtained in Lemmas \ref{lemm:mqinn}, \ref{lemm:mep_mid} and \ref{lemm:mep_out} below.

\section{Control of the solution in the bootstrap regime} \label{sec:4}

We now fix $\bar M_0>0$ and pick constants $A,\eta,\kappa,\tau_0$ and $K>1$ whose values are allowed to change from one lemma to another. When proving Proposition \ref{prop:trap_mq} at the end of the section, we will prove that the conclusions of all lemmas are simultaneously valid for values of $A,K,\eta,\kappa,\tau_0$ as described in the proposition.

Throughout the section, we consider a solution $m_\varepsilon$ to \eqref{eq:mep0} with data $m_\varepsilon(0)$ that satisfies \eqref{bd:bootstrap:initial} and \eqref{bd:bootstrap:initialregularity}, with $R(\tau_0)=e^{-\tau_0/2}$ and $M(\tau_0)=M_0$. We assume that for some $t_1>0$, there exist $R,M\in C^1([0,t_1],(0,\infty))$ such that, defining $\tau$ by \eqref{def:mw}, then $m_\varepsilon(\tau)\in \Sc (\tau(t))$ for all $\tau \in [\tau_0,\tau_1]$ where $\tau_1=\tau(t_1)$, and that the parameters $R$ and $M$ given by Definition \ref{def:bootstrap} coincide with $R(\tau(t))$ and $M(\tau(t))$. We pick any $s_0\in \mathbb R$, define $s$ by \eqref{def:xis} and introduce $s_1=s(t_1)$.

Note that for $\tau_0$ large enough, there exists $t_1>0$ such that this holds true and that $M$ and $R$ are unique, as a consequence of the continuity of the flow of \eqref{eq:mep0} and of the implicit function Theorem to determine $M$ and $R$ from the orthogonality conditions \eqref{orthogonality2} and \eqref{orthCon2}. We omit the proof of this standard fact.

\subsection{A priori bounds}

\begin{lemma} \label{lem:aprioriestimates}

There exists $A^*>0$ such that for any $A\geq A^*$, for any $\kappa,\eta,\bar M_0>0$ and $K\geq e^{3A/10}$, if $\tau_0$ is large enough, then for $\tau_0\leq \tau \leq \tau_1$:
\be \label{bd:pointwisemq}
|m_q (s,\xi)|\lesssim \left\{
\begin{array}{l l l l}
Ke^{-\kappa \tau}e^{-\frac{|\xi|}{4}} &\qquad  \mbox{for }|\xi|\leq 4|\log \nu|+A,\\
\nu K^{\frac 54}e^{-\frac 38 A} e^{-\kappa \tau}(1+|\xi-\xi_{A,+}|)\zeta^{d-1}&\qquad  \mbox{for } \xi>  \xi_{A,+},\\
\nu K^{\frac 54}e^{-\frac 38 A} e^{-\kappa \tau}(\hat \chi_\eta(\zeta)+\zeta^d) & \qquad \mbox{for } \xi <\xi_{A,-}.
\end{array}
 \right.
\ee
\be \label{bd:boundarynorm}
\| m_\varepsilon \|_{\bd}\lesssim \nu K^{\frac{5}{4}} e^{-\frac 38 A} e^{-\kappa \tau},
\ee
\be \label{bd:estimationtau}
\frac{4^{1-d}}{M_0} e^{-\frac{d-2}{2}\tau}\leq \nu \leq \frac{4^{d-1}}{M_0} e^{-\frac{d-2}{2}\tau}.
\ee
\end{lemma}

\begin{proof}

The first inequality in \eqref{bd:pointwisemq} is obtained from the Sobolev estimate \eqref{iq:LinfSobolev}, \eqref{eq:normequiv} and \eqref{est:mqL2ome}. The second inequality in \eqref{bd:pointwisemq} is obtained from \eqref{est:mqbydev} and \eqref{est:mepLinf_out} using that $\pa_\xi=\nu \pa_\zeta$ and $1\leq K^{5/4}e^{-3A/8}$ as $K\geq e^{3A/10}$. Then, we estimate that for $1-2\eta\leq \zeta\leq \zeta_{A,-}$:
\be \label{bd:technicalapriori1}
K^{\frac 54}\int_0^\zeta e^{-\frac 38 \frac{\zeta_--\tilde \zeta}{\nu}}\hat \chi_\eta (\tilde \zeta)  d\tilde \zeta \lesssim K^{\frac 54} \int_0^\zeta e^{-\frac 38 \frac{\zeta_--\tilde \zeta}{\nu}}  d\tilde \zeta \lesssim K^{\frac 54} \nu e^{-\frac 38 \frac{\zeta_--\tilde \zeta}{\nu}} \lesssim \nu K^{\frac 54} e^{-\frac 38 A}\hat \chi_\eta+\nu \zeta^d,
\ee
where we used \eqref{def:hatchi_eta}, that $e^{-\frac 38 \frac{\zeta_--\tilde \zeta}{\nu}}\leq e^{-\frac 38 A }$ for $1-\eta \leq \zeta\leq \zeta_{A,-}$ and $K^{\frac 54} e^{-\frac 38 \frac{\zeta_--\tilde \zeta}{\nu}}\leq K^{\frac 54} e^{-\frac{\eta}{4 \nu} }\lesssim 1$ for $1-2\eta \leq \zeta\leq 1 -\eta$ for $\tau_0$ large enough depending on $K$. The third inequality in \eqref{bd:pointwisemq} is then obtained from \eqref{est:mepLinf_out2} using $m_\varepsilon(0)=0$ and \eqref{bd:technicalapriori1}. Then, \eqref{bd:boundarynorm} is a direct consequence of \eqref{bd:pointwisemq}, \eqref{est:mepLinf_out} and \eqref{est:mepLinf_out2}. Finally, \eqref{bd:estimationtau} follows from \eqref{eq:relRMnu} and \eqref{est:MR}. \end{proof}

\subsection{Modulation equations}

The evolution of the modulation parameters $R$ and $M$ is given in the following lemma. 

\begin{lemma}[Modulation equations] \label{lemm:mod}
There exists $A^*>0$ such that for any $A\geq A^*$, for any $\kappa,\eta,\bar M_0>0$ and $K\geq e^{3A/10}$, for $\tau_0$ large enough, there holds for all $\tau_0\leq \tau\leq \tau_1$,
\begin{align}\label{eq:ModR}
& \left| \frac{R_\tau}{ R} + \frac{1}{2} \right|  \lesssim  \nu +  e^{-\frac{A}{4}}\|m_q(\tau)\|_{\inn} +A \| m_\varepsilon \|_{\bd} ,\\
& \label{eq:ModM} \left| \frac{M_s}{M}\right| \lesssim  \nu^3 |\log \nu|e^{-\frac A2}+\nu^2 e^{-\frac 34 A}\|m_q(\tau)\|_{\inn} + \|  m_\varepsilon \|_{\bd}.
\end{align}
\end{lemma}

\begin{corollary} \label{cor:mod}

There exists $\kappa^*(d)>0$ such that for $0<\kappa\leq \kappa^*$ and under the assumptions of Lemma \ref{lemm:mod}, for $\tau_0$ large enough we have:
\begin{align}\label{eq:Modnu}
& \left| \nu_\tau+\frac{d-2}{2}\nu \right|\lesssim \nu \left( \nu+e^{-\frac A4}\| m_q\|_{\inn}+\nu^{-1} \| m_\varepsilon \|_{\bd}  \right),\\
\label{bootstrap:improvedRM}
& \frac{1}{2} e^{-\frac{\tau}{2}}\leq R(\tau) \leq 2e^{-\frac{\tau}{2}}, \qquad \frac{M_0}{2}\leq M(\tau) \leq 2M_0 .
\end{align}
Moreover, if $m_u$ is trapped in $\Sc(\tau)$ for all $\tau \in [\tau_0, +\infty)$, there exist $\tilde R_\infty,M_\infty>0$ so that
\begin{align}
\label{modulation:bd:asymptoticR} & R(\tau)=\tilde R_\infty e^{-\frac \tau 2}\big(1+\Oc(e^{-\kappa \tau})\big), \\
\label{modulation:bd:asymptoticM} & M(\tau)=M_\infty \big(1+\Oc(e^{-\kappa \tau})\big),\\
\label{modulation:bd:asymptoticnu} & \nu(\tau)=\tilde \nu_\infty e^{-\frac{d-2}{2} \tau }\big(1+\Oc(e^{-\kappa \tau})\big),
\end{align}
where $\tilde \nu_\infty=\tilde R_{\infty}^{d-2}M_\infty^{-1}$, and where the constants in the $\Oc()$ depend on $K,\kappa,\bar M_0$.

\end{corollary}

\begin{proof}[Proof of corollary  \ref{cor:mod}]

Recall $\nu=R^{d-2}/M$ and $M_0\geq \bar M_0$. We obtain the inequality \eqref{eq:Modnu} by combining \eqref{eq:ModR} and \eqref{eq:ModM}. Then, injecting \eqref{est:MR}, \eqref{bd:estimationtau}, \eqref{est:mqL2ome} and \eqref{bd:boundarynorm} into \eqref{eq:ModR} yields
\be
|\frac{d}{d\tau} (e^{\frac \tau 2}R)|\leq C(K,\bar M_0) (e^{-\frac{d-2}{2}\tau}+e^{-\kappa \tau}+e^{-\frac{d-2}{2}\tau}e^{-\kappa \tau}) \leq C(K,\bar M_0)  e^{- \kappa \tau},
\ee
for $\kappa\leq \frac{d-2}{2}$. Integrating between $\tau_0$ and $\tau$ using \eqref{bootstrap:choiceR0}:
$$
R(\tau)=e^{-\frac \tau 2}\left(e^{\frac{\tau_0}{2}}R(\tau_0)+\int_{\tau_0}^\tau \Oc_{K,\bar M_0}(e^{-\kappa \tilde \tau})d\tilde \tau\right)=e^{-\frac{\tau}{2}}(1+\Oc_{K,\bar M_0,\kappa}(e^{- \kappa \tau_0})).
$$
This gives the first inequalities in \eqref{bootstrap:improvedRM} upon choosing $\tau_0$ large depending on $\kappa,\bar M_0,K$. The second inequalities in \eqref{bootstrap:improvedRM} are obtained similarly using $M_\tau=\nu^{-1}M_s$. Then, if $m_q$ is trapped in $\Sc(\tau)$ for all $\tau \in [\tau_0, +\infty)$, we rewrite the above identity as
$$
R(\tau)=e^{-\frac \tau 2}\Bigl(\underbrace{e^{\frac{\tau_0}{2}}R(\tau_0)+\int_{\tau_0}^\infty \Oc_{K,\bar M_0}(e^{-\kappa \tilde \tau})d\tilde \tau}_{=\tilde R_\infty} -\int_\tau^\infty \Oc_{K,\bar M_0}(e^{-\kappa \tilde \tau})d \tilde \tau \Bigr)=e^{-\frac{\tau}{2}}(\tilde R_\infty+\Oc(e^{-\kappa \tau})),
$$
where the constant in the last $\Oc()$ depends on $\kappa,K,\bar M_0$. This results in \eqref{modulation:bd:asymptoticR}. The inequality \eqref{modulation:bd:asymptoticM} is obtained similarly using \eqref{eq:ModM}, and \eqref{modulation:bd:asymptoticnu}  follows from \eqref{modulation:bd:asymptoticR} and \eqref{modulation:bd:asymptoticM} as $\nu=R^{d-2}M^{-1}$. This ends the proof of the Corollary.
\end{proof}

\begin{proof}[Proof of Lemma \ref{lemm:mod}]

\textbf{Step 1}. \emph{Computation of $R$}. We claim that for $\tau$ large enough,
\be \label{mod:bd:Rtauinter}
\left| \frac{R_\tau}{ R} + \frac{1}{2} \right|  \lesssim   e^{-\frac{A}{4}}\|m_q\|_{\inn}  + \nu +A \left| \frac{M_s}{M} \right|.
\ee
To show \eqref{mod:bd:Rtauinter}, we differentiate \eqref{orthogonality2} with respect to $s$, use equation \eqref{eq:mqxis} and the localization of $\chi_{_A}$ to get 
\be \label{mod:id:Rtaupreliminary}
0 = \int_{-1/\nu}^{+\infty} \Big[ \Ls_0 (m_q) + L(m_q) + \frac{m_q \partial_\xi m_q}{(1 +\nu \xi )^{d-1}} + \Psi \Big] \chi_{_A} Q' \omega_0 d\xi.
\ee
We now compute the contribution of all terms above. For the first one, using $\Ls_0(Q') = 0$ we obtain $|\Ls_0 ( \chi_{_A})Q']| \lesssim A^{-1}e^{-|\xi|/2} \mathbf{1}_{\{A \leq |\xi| \leq 2A\}}$, so that using that $\Ls_0$ is self-adjoint in $L^2_{\omega_0}$:
\begin{align}
\nonumber &\left |\int_{-1/\nu}^{+\infty} \Ls_0 m_q \chi_{_A} Q' \omega_0 d\xi \right| = \left|\int_{-1/\nu}^{+\infty}  m_q \Ls_0 \Big[ \chi_{_A} Q' \Big]\omega_0 d\xi \right| \lesssim A^{-1}\int_{A\leq |\xi|\leq 2A} |m_q| e^{-\frac{|\xi|}{2}} \omega_0 d\xi  \\ 
 \label{mod:bd:inter1} & \quad \lesssim A^{-1}\left(\int_{ A\leq |\xi| \leq 2A} |m_q|^2 \omega_0 d\xi \right)^{\frac{1}{2}} \left(\int_{ A\leq |\xi|\leq 2A} e^{-|\xi|}\omega_0 d\xi \right)^{\frac{1}{2}} \lesssim A^{-1}e^{-\frac{A}{4}}\| m_q\|_{\inn},
\end{align}
where we used that $\omega_0 \lesssim e^{|\xi|/2}$ and \eqref{eq:normequiv} (valid for $A$ large enough). For the second, since $\bar Q=Q$ for $|\xi|\leq 4|\log \nu|+A+1$, one has 
\begin{align*}
 L(m_q)&= -(d-1)\nu \left(\frac 12 \xi +\frac{1}{1+\nu \xi} \right) \pa_\xi m_q + \left(\frac{R_\tau}{R}+\frac{1}{2} \right)\left(1+(d-1)\nu \xi \right)\pa_\xi m_q-\frac{M_s}{M}\left(m_q+\xi\pa_\xi m_q\right)\\
 &\qquad +\left(\frac{1}{(1+\xi\nu)^{d-1}}-1 \right)\bar Q \pa_\xi m_q+\left(\frac{1}{(1+\xi\nu)^{d-1}}-1 \right)m_q\pa_\xi \bar Q
\end{align*}
and hence for $|\xi|\leq \xi_{A,+}+1$, we have the rough estimate
\be \label{mod:bd:pointwiseL}
|L(m_q)| \lesssim \left( \left| \frac{R_\tau}{R}+\frac 12 \right|+\nu \langle \xi \rangle +\frac{|M_s|}{M}|\xi|\right)\pa_\xi m_q+\left(\frac{|M_s|}{M}+\nu|\xi|e^{-\frac{|\xi|}{2}}\right)|m_q|,
\ee
provided that $\nu$ is small enough, i.e. that $\tau_0$ is large enough depending on $\bar M_0$ from \eqref{bd:estimationtau}. Using \eqref{mod:bd:pointwiseL}, $|\pa_\xi Q | \lesssim e^{-|\xi|/2}$ and \eqref{eq:normequiv}  we estimate
\be  \label{mod:bd:inter2}
\left| \int_{-1/\nu}^{+\infty} L(m_q) \chi_{_A} Q' \omega_0 d\xi\right| \lesssim \left(  \left| \frac{R_\tau}{R} + \frac{1}{2}  \right| + \left| \frac{M_s}{M}\right|+\nu \right) \|m_q\|_{\inn} . 
\ee
The nonlinear term is estimated by Cauchy-Schwarz and \eqref{eq:normequiv},
\begin{equation}  \label{mod:bd:inter3}
\left| \int_{-1/\nu}^{+\infty} \frac{m_q \partial_\xi m_q}{(1 + \nu \xi)^{d-1}} \chi_{_A} Q' \omega_0 d\xi  \right| \lesssim \left| \int m_q^2\chi_A \omega_0  \right|^{\frac 12}\left| \int |\pa_\xi m_q|^2\chi_A \omega_0  \right|^{\frac 12} \lesssim \|m_q\|^2_{\inn}.
\end{equation}
Finally, for the error term, as $\bar \chi=1$ for $|\xi|\leq \xi_{A,+}+1$ we compute that there:
\begin{align}
\nonumber \Psi(\xi,s) &=  \left(\frac{R_\tau}{R}+\frac{1}{2} \right)\left(1+(d-1)\nu \xi \right)\pa_\xi Q-\frac{M_s}{M}\left( Q+\xi \pa_\xi  Q\right)  -(d-1)\nu \left(\frac 12 \xi +\frac{1}{1+\nu \xi} \right) \pa_\xi  Q \\
 \label{mod:bd:pointwisePsi} &\qquad +\left(\frac{1}{(1+\xi\nu)^{d-1}}-1 \right) Q \pa_\xi  Q -\nu \frac{d-1}{1+\nu\xi}\pa_\xi  Q 
\end{align}
so that using $Q\leq 1$ and $|\pa_\xi Q|\lesssim e^{-|\xi|/2}$, we obtain
\be  \label{mod:bd:inter4}
\int_{-1/\nu}^{+\infty} \Psi \chi_{_A}Q' \omega_0 d\xi = \left( \frac{R_\tau}{R} + \frac{1}{2}\right)\left( \int_{-1/\nu}^{+\infty} [Q']^2 \chi_{_A}\omega_0 d\xi  +O(\nu)\right) +  \; \Oc\left(A\left| \frac{M_s}{M} \right| +\nu \right).
\ee
Injecting \eqref{mod:bd:inter1}, \eqref{mod:bd:inter2}, \eqref{mod:bd:inter3}, \eqref{mod:bd:inter4} in \eqref{mod:id:Rtaupreliminary}, using \eqref{est:mqL2ome} shows \eqref{mod:bd:Rtauinter} for $\tau_0$ large enough.\\

\noindent \textbf{Step 2}. \emph{Computation of $M$}. We claim the following:
\be \label{mod:bd:Msinter}
\left| \frac{M_s}{M} \right|  \lesssim \nu \| \pa_\zeta m_\varepsilon \|_{\bd}+\left|\frac{R_\tau}{R}+\frac 12 \right| \left(\nu \| \pa_\zeta m_\varepsilon \|_{\textup{bou}}+\nu^2e^{-A/2} \right)+\nu^3 |\log \nu|e^{-A/2}.
\ee
To show it, we differentiate in time the orthogonality condition \eqref{orthCon2} and use the equation \eqref{eq:mqxis} to write 
\be \label{mod:id:Rtaupreliminary2}
0 = \int_{-1/\nu}^{+\infty} \Big[ \Ls_0 m_q + L(m_q) + \frac{m_q \partial_\xi m_q}{(1 +\nu \xi )^{d-1}} + \Psi(\xi,s)\Big] \bar \chi d\xi - \int_{-1/\nu}^{+\infty} m_q \pa_s \bar \chi  d\xi, 
\ee
where we write for short in this proof $\bar \chi= \bar \chi_{_{1,\xi_{A,+}}} (\xi)$. Recall that $\textup{supp}\big(\bar \chi \big) \subset (\xi_{A,+}-2, \xi_{A,+}+2)$. Using \eqref{def:Lop}, integrating by parts, and then using \eqref{est:mqbydev} and $\pa_\xi=\nu \pa_\zeta$, we estimate
\begin{align}
\nonumber \left|\int_{-1/\nu}^{\infty} \Ls_0 m_q  \bar \chi d\xi\right| & \lesssim \int_{-1/\nu}^{+\infty} |\pa_\xi m_q| \big(\big|\pa_\xi \bar \chi \big|  + \bar \chi \big)+  |m_q|| \pa_\xi Q| \bar \chi \Big) d\xi \\
 \label{mod:bd:inter11} & \qquad   \lesssim \|\partial_\xi m_q\|_{L^\infty(\xi_{A,+}-2,\xi_{A,+}+2)} \ = \  \|  m_\varepsilon \|_{\bd}.
\end{align}
Using \eqref{mod:bd:pointwiseL}, \eqref{est:mqbydev} and $\textup{supp}\big(\bar \chi \big) \subset (\xi_{A,+}-2, \xi_{A,+}+2)$, we get that:
\be  \label{mod:bd:inter12}
 \left|\int_{-1/\nu}^{+\infty} L(m_q)  \bar \chi d\xi\right| \lesssim \|  m_\varepsilon \|_{\bd} \left(  \left| \frac{R_\tau}{R} + \frac{1}{2} \right| + \nu |\log \nu|+|\frac{M_s}{M}||\log \nu| \right).
\ee
For the nonlinear term, we have by \eqref{est:mqbydev},
\be  \label{mod:bd:inter13}
\left|\int_{-1/\nu}^{+\infty} \frac{m_q \pa_\xi m_q}{(1 + \nu \xi)^{d-1}}  \bar \chi d\xi\right| \lesssim  \|  m_\varepsilon \|_{\bd}^2.
\ee
As $Q=1+O(e^{-|\xi|/2})$ and $|\pa_\xi Q|\lesssim e^{-|\xi|/2}$, we use \eqref{mod:bd:pointwisePsi} and $\textup{supp}\bar \chi$ to write
$$
\Psi(s,\xi)= -\frac{M_s}{M}\left(1+O(\nu^2|\log \nu|e^{-A/2})\right)+O\left(|\frac{R_\tau}{R}+\frac 12|\nu^2 e^{-A/2} \right)+O(\nu^3 |\log \nu|e^{-A/2}).
$$
From the above identity, we deduce
\be  \label{mod:bd:inter14}
\int_{-1/\nu}^{\infty} \Psi\bar \chi d\xi = -\frac{M_s}{M}\left( \int_{\mathbb R} \chi d\xi +\Oc(\nu^2 |\log \nu|e^{-\frac A2})\right)+\Oc\left(\nu^2 e^{-\frac A2} (|\frac{R_\tau}{R}+\frac12 |+\nu |\log \nu|)\right).
\ee
Using \eqref{eq:relRMnu} and $\big|\pa_s \bar \chi \big| \leq \left|\nu_\tau\right|\mathbf{1}_{ \xi_{A,+}-2\leq \xi \leq \xi_{A,+}+2}$, \eqref{est:mqbydev}, we estimate 
\be  \label{mod:bd:inter15}
 \left|\int_{-1/\nu}^{+\infty} m_q \pa_s \bar \chi d\xi \right| \lesssim \nu \|  m_\varepsilon \|_{\bd} \left(1+ \left|\frac{R_\tau}{R}+\frac 12 \right|+\left|\frac{M_s}{M}\right| \right).
\ee
Injecting \eqref{mod:bd:inter11}, \eqref{mod:bd:inter12}, \eqref{mod:bd:inter13}, \eqref{mod:bd:inter14} and \eqref{mod:bd:inter15} in \eqref{mod:id:Rtaupreliminary2}, using that $\int_{\mathbb R}\chi d\xi>0$ and $ |\log \nu| \| m_\varepsilon \|_{\textup{bou}}\rightarrow 0$ as $\tau_0\rightarrow \infty$ from \eqref{est:mepLinf_out}, shows \eqref{mod:bd:Msinter}.\\

\noindent \textbf{Step 3}. \emph{End of the proof}. Combining \eqref{mod:bd:Msinter} and \eqref{mod:bd:Rtauinter} shows \eqref{eq:ModR} and \eqref{eq:ModM}.
\end{proof}

\subsection{Improved $\|m_q\|_\inn$ bound}
The following lemma shows that $\|m_q\|_\inn$ is a Lyapunov functional in the trapped regime.
\begin{lemma}[Monotonicity of $\|m_q\|_\inn$]\label{lemm:mqinn}  There exist $\delta_2 > 0$ and $C > 0$ such that the following holds. There exists $A^*>0$ such that for any $A\geq A^*$, for any $\kappa,\eta,\bar M_0>0$ and $K\geq e^{3A/10}$, for $\tau_0$ large enough, for all $s_0\leq s \leq s_1$:
\begin{align} \label{bd:innerlyapunov}
\frac{d}{ds}\|m_q(s)\|^2_\inn &\leq -\delta_2 \| m_q(s)\|^2_\inn + Ce^{\frac A 2}\nu^{-2} \| m_\varepsilon (\tau) \|_{\bd}^2 + C\nu^2.
\end{align}
\end{lemma}

\begin{proof} In this part we shall write $\chi = \chi_{_{A}}^\inn$ introduced in \eqref{def:chi_inn} for sake of simplicity. We obtain from \eqref{eq:mqxis}, from the commutator relation
$$
\Ls_0 (\chi m_q)=\chi\Ls_0m_q+2\pa_\xi \chi \pa_\xi m_q+\left(\pa_\xi^2 \chi-(\frac 12-Q)\pa_\xi \chi\right)m_q,
$$
and from the self-adjointness of $\Ls_0$ in $L^2_{\omega_0}$, the energy identity
\begin{align}
\nonumber \frac{1}{2}\frac{d}{ds}\| m_q(s)\|^2_\inn &= -\int_{-\nicefrac{1}{\nu}} ^\infty \Ls_0 m_q^\inn\Big[ \Ls_0 m_q^\inn + \left(\frac 12  - Q \right)\pa_\xi \chi m_q - \pa_\xi^2\chi m_q - 2 \pa_\xi m_q \pa_\xi \chi + \pa_s \chi m_q  \Big]\omega_0 d\xi\\
\label{inner:id:energy} &\quad - \int_{-\nicefrac{1}{\nu}}^\infty  \Ls_0 m_q^\inn\Big[ L(m_q)\chi +\chi \frac{m_q \pa_\xi m_q}{(1  +\nu \xi)^{d-1}} + \Psi \chi \Big]  \omega_0 d\xi.
\end{align}

\paragraph{The linear term.} Since $m_q^\inn$ has compact support within $(-\nu^{-1},\infty)$, we may extend $m_q^\inn$ by $0$ for $\xi\leq -\nu^{-1}$ in order to apply Lemma \ref{lemm:coerL02}. Using \eqref{orthogonality2} and \eqref{def:chi_inn} we obtain $\int_{\mathbb R}m_q^\inn \pa_\xi Q \chi_A \omega_0d\xi=0$. Applying \eqref{coer:id:coercivity} and using \eqref{def:Lop} yield
\be \label{inner:bd:linear}
\int_{-\nicefrac{1}{\nu}} ^\infty |\Ls_0 m_q^\inn|^2\omega_0 d\xi \geq \delta_1 \| m_q^\inn\|_{H^2_{\omega_0}}^2\geq \bar \delta \| m_q\|_{\inn}^2, \quad \textup{for some $\bar \delta>0$.}
\ee

\paragraph{The boundary terms.} By definition of $\chi$ and using \eqref{eq:Modnu} (implying $|\nu_\tau|\lesssim \nu$), we have 
\be \label{inner:bd:boundary}
\big|\pa_\xi^k \chi\big| \lesssim \mathbf{1}_{\{(\xi_{A,+} \leq |\xi|\leq \xi_{A,+} + 1)\}}, \quad \big|\pa_s \chi \big| \lesssim |\nu| \mathbf{1}_{\{(\xi_{A,+} \leq |\xi| \leq  \xi_{A,+} + 1)\}}.
\ee
Note that
\be \label{inner:id:omega0boundary}
\omega_0(\xi)\approx \nu^{-2}e^{\frac A2} \qquad \mbox{for}\quad \xi_{A,+} \leq |\xi| \leq \xi_{A,+} + 1,
\ee
we then estimate by using the two above inequalities, \eqref{def:boundarynorm} and $\pa_\xi=\nu \pa_\zeta$,
\be
\label{inner:bd:boundaryterm} \left|\int_{-\nicefrac{1}{\nu}} ^\infty \Big[ \left( \frac 12 - Q \right)\pa_\xi \chi m_q - \pa_\xi^2\chi m_q - 2 \pa_\xi m_q \pa_\xi \chi+\pa_s \chi m_q \Big]^2\omega_0 d\xi\right|  \lesssim \nu^{-2} e^{\frac A2} \| \pa_\zeta m_\varepsilon \|_{\bd}^2.
\ee

\paragraph{The generated error term.} We recall from \eqref{mod:bd:pointwisePsi} that for $|\xi|\leq \xi_{A,+}+1$,
\begin{align}  \label{inner:id:Psi}
\Psi &=  \left(\frac{R_\tau}{R}+\frac{1}{2} \right)\pa_\xi Q-\frac{M_s}{M}\left( Q+\xi \pa_\xi Q\right) +\tilde \Psi,\\
\nonumber \tilde \Psi &= \left(\frac{R_\tau}{R}+\frac{1}{2} \right) (d-1)\nu \xi \pa_\xi Q  -(d-1)\nu \left(\frac 12 \xi +\frac{1}{1+\nu \xi} \right) \pa_\xi  Q \\
&\qquad +\left(\frac{1}{(1+\xi\nu)^{d-1}}-1 \right) Q \pa_\xi  Q -\nu \frac{d-1}{1+\nu\xi}\pa_\xi  Q .
\end{align}
For the first term, we use the fact that $\Ls_0$ is self-adjoint in $L^2_{\omega_0}$, $\Ls_0 \pa_\xi Q = 0$ and \eqref{def:Lop}, then Cauchy-Schwarz, \eqref{inner:bd:boundary}, $|\pa_\xi Q|\lesssim e^{-|\xi|/2}$, and $\omega_0\approx e^{|\xi|/2}$ to write
\begin{align}
\nonumber &\left|\left( \frac{R_\tau}{R} + \frac{1}{2} \right)\int_{-\nicefrac{1}{\nu}}^\infty \Ls_0 m_q^\inn \pa_\xi Q \chi \omega_0 d\xi \right| \\
\label{inner:bd:inter11} &\quad =\left|\left( \frac{R_\tau}{R} + \frac{1}{2} \right)\int_{-\nicefrac{1}{\nu}}^\infty m_q^\inn \left((\pa_\xi^2\chi-(\frac 12 -Q)\pa_\xi \chi)\pa_\xi Q+2\pa_\xi \chi \pa_\xi^2 Q\right)\omega_0 d\xi \right|\\
\nonumber &\quad  \lesssim \left|\frac{R_\tau}{R} + \frac{1}{2} \right| \|m_q\|_\inn  \left(\int_{\xi_{A,+}}^{\xi_{A,+}+1} e^{-\frac{|\xi|}{2}}d\xi \right)^\frac{1}{2} \ \lesssim \ \nu e^{-\frac A4} \left|\frac{R_\tau}{R} + \frac{1}{2} \right| \|m_q\|_\inn \ \lesssim \ \nu  \|m_q\|_\inn,
\end{align}
where we used the rough estimate $e^{-A/4}|R_\tau/R+1/2|\lesssim 1$ from \eqref{eq:ModR} for the last inequality. For the second term, using the self-adjointness of $\Ls_0$, \eqref{def:Lop}, then Cauchy-Schwarz, $|\Ls_0 Q|\lesssim e^{-|\xi|/2}$, \eqref{inner:bd:boundary}, $Q=1+O(e^{-|\xi|/2})$, $|m_q^\inn|\lesssim |m_q|$ and $\omega_0\approx e^{|\xi|/2}$ and \eqref{est:mqbydev} yields
\begin{align*}
& \left|  \intb  \Ls_0 m_q^\inn (Q+\xi\pa_\xi Q) \chi \omega_0 \right| = \Biggl| \intb  m_q^\inn \Ls_0(Q+\xi\pa_\xi Q) \chi \omega_0d\xi \\
&\quad  \qquad + \int_{-\nicefrac{1}{\nu}}^\infty m_q^\inn \left((\pa_\xi^2\chi-(\frac 12 -Q)\pa_\xi \chi)(Q+\xi\pa_\xi Q) +2\pa_\xi \chi (2\pa_\xi Q+\xi\pa_\xi^2 Q)\right)\omega_0 d\xi \Biggr|\\
&\quad \lesssim  \|m_q^\inn \|_{L^2_{\omega_0}}+\int_{\xi_{A,+}\leq |\xi| \leq \xi_{A,+}+1} |m_q(\xi)|e^{\frac{|\xi|}{2}} d\xi \ \lesssim \  \|m_q\|_\inn + \| m_\varepsilon \|_{\bd} \nu^{-2}e^{\frac A2}. 
\end{align*}
Using \eqref{eq:ModM}, \eqref{est:mqL2ome} and \eqref{bd:boundarynorm} (so that $\nu^{-1}\|  m_\varepsilon \|_{\bd} \lesssim 1 $), we get the rough bound $ |\frac{M_s}{M}|\lesssim \nu $. Thus, 
$$
|\frac{M_s}{M}| \|m_q\|_\inn \lesssim \nu  \|m_q\|_\inn.
$$
Second, using \eqref{eq:ModM} and the inequality $xy\leq x^2/2+y^2/2$ yields
\begin{align*}
 |\frac{M_s}{M}| \| \pa_\zeta m_\varepsilon \|_{\bd} \nu^{-1}e^{\frac A2} & \lesssim \nu |\log \nu|\|  m_\varepsilon \|_{\bd}+e^{-A/4} \|  m_\varepsilon \|_{\bd} \| m_q\|_{\inn}+\nu^{-2} \|  m_\varepsilon \|_{\bd}^2 e^{\frac A2} \\
& \lesssim \nu^4 |\log \nu|^2 e^{-\frac A2}+e^{-A}\nu^2  \| m_q\|_{\inn}+\nu^{-2} \|  m_\varepsilon \|_{\bd}^2 e^{\frac A2} .
\end{align*}
We conclude by using the three previous inequalities,
\be \label{inner:bd:inter12}
 \left| \frac{M_s}{M} \intb  \Ls_0 m_q^\inn Q \chi \omega_0 \right|  \lesssim \nu \|m_q\|_\inn + \nu^4 |\log \nu|^2 e^{-\frac A2}+\nu^{-2} \|  m_\varepsilon \|_{\bd}^2 e^{\frac A2}.
 \ee
To estimate the remaining term, using \eqref{eq:ModR} and $|\pa_\xi Q|\lesssim e^{-|\xi|/2}$ we obtain $|\tilde \Psi (s,\xi)|\lesssim \nu \langle \xi \rangle e^{-|\xi|/2}$. Hence, we have by Cauchy-Schwarz and $\omega_0\approx e^{|\xi|/2}$,
$$
\left| \intb \Ls_0 m_q^\inn \tilde \Psi \chi \omega_0 \right| \lesssim \| \Ls_0 m_q^\inn \|_{L^2_{\omega_0}} \| \tilde \Psi \chi \|_{L^2_{\omega_0}}\lesssim \nu  \| \Ls_0 m_q^\inn \|_{L^2_{\omega_0}}.
$$
Injecting \eqref{inner:bd:inter11}, \eqref{inner:bd:inter12} and the above inequality in \eqref{inner:id:Psi}, then using \eqref{inner:bd:linear} and $xy\leq \mu x^2/2+\mu^{-1} y^2/2$ shows that
\begin{align}
 \label{inner:bd:interPsi}  & \left|  \intb  \Ls_0 m_q^\inn \Psi \chi \omega_0 \right|\leq C\left(  \nu  \| \Ls_0 m_q^\inn \|_{L^2_{\omega_0}}+ \nu^4 |\log \nu|^2 e^{-\frac A2}+\nu^{-2} \|  m_\varepsilon \|_{\bd}^2 e^{\frac A2} \right)\\
\nonumber & \leq  C\mu \| \Ls_0 m_q^\inn \|_{L^2_{\omega_0}}^2+C\mu^{-1}\nu^2+\nu^{-2} \|  m_\varepsilon \|_{\bd}^2 e^{\frac A2}  \ \leq \ \frac{1}{10} \| \Ls_0 m_q^\inn \|_{L^2_{\omega_0}}^2+C \nu^2+\nu^{-2} \|  m_\varepsilon \|_{\bd}^2 e^{\frac A2},
\end{align}
if $\mu>0$ has been chosen small enough.\\

\paragraph{The small linear term and the nonlinear term:} We first estimate using \eqref{mod:bd:pointwiseL}, \eqref{eq:ModR}, \eqref{eq:ModM} and \eqref{bd:pointwisemq} that for $|\xi|\leq \xi_{A,+}+1$:
$$
L(m_q)+ \frac{m_q \pa_\xi m_q}{(1 + \nu\xi)^{d-1}} =o(|\pa_\xi m_q|)+o(|m_q|),
$$
where the $o()$ is as $\tau_0\to \infty$, and is uniform for $|\xi|\leq \xi_{A,+}+1$. Hence, using the above inequality, then the decomposition \eqref{def:chi_inn}, and then \eqref{inner:id:omega0boundary} and $\omega_0(\xi)\approx e^{|\xi|^2/2}$:
\begin{align*}
\int_{-1/\nu}^\infty \chi^2 |L(m_q)+ \frac{m_q \pa_\xi m_q}{(1 + \nu\xi)^{d-1}}|^2\omega_0 d\xi & = \int_{|\xi|\leq \xi_{A,+}}^\infty...+\int_{\xi_{A,+} \leq |\xi|\leq \xi_{A,+}+1 }^\infty... \\
 &=o(\| m_q \|_{\inn}^2)+o(\nu^{-2}\|  m_\varepsilon \|_{\bd}^2 e^{A/2}).
\end{align*}
We thus obtain by using Cauchy-Schwarz, the above inequality and then \eqref{inner:bd:linear},
\begin{align}
 \label{inner:bd:inter5} & \left| \intb \Ls_0 m_q^\inn \left(\chi L(m_q)+ \chi  \frac{m_q \pa_\xi m_q}{(1 + \nu\xi)^{d-1}}\right)\omega_0 d\xi \right| \\
 \nonumber &\quad  \lesssim  \|\Ls_0 m_q^\inn\|_{L^2_{\omega_0}}\left(o(\| m_q \|_{\inn})+o( \nu^{-1}e^{A/4}\|  m_\varepsilon \|_{\bd} \right) \ = \ o\left( \|\Ls_0 m_q^\inn\|_{L^2_{\omega_0}}^2+\nu^{-2}e^{A/2} \| m_\varepsilon \|_{\bd}^2 \right).
\end{align}

\paragraph{Conclusion:} Injecting \eqref{inner:bd:linear}, \eqref{inner:bd:boundaryterm}, \eqref{inner:bd:interPsi} and \eqref{inner:bd:inter5} in \eqref{inner:id:energy} shows \eqref{bd:innerlyapunov} and concludes the proof of Lemma \ref{lemm:mqinn}.
\end{proof}

\subsection{Improved exterior bound}
In this subsection we improve the bootstrap bounds \eqref{est:mepLinf_out}  and \eqref{est:mepLinf_out2}. We first study the exterior zone $\zeta\geq \zeta_+=1+4\nu|\log \nu|$ (or $\xi\geq \xi_+=4|\log \nu|$). We have from \eqref{eq:phisfrsft}  $\bar Q_\nu(\zeta) =Q_\nu(\zeta) = 1 + \Oc(\nu^2)$ for $\zeta \geq \zeta_+$. We write the equation satisfied by $m_\ep$ as a linear equation: 
\begin{equation}\label{eq:mep}
\pa_\tau m_\ep = \As m_\ep + \Pc m_\ep + E \qquad \textup{for}\quad  \zeta \geq \zeta_+,
\end{equation}  
where the main order operator $\As$ and the lower order operator $\Pc $ (note that $\Pc $ depends on $m_\varepsilon$, i.e. we are including nonlinear transport terms in the operator $\Pc $) are
\begin{align}
\As &= \left( \frac{1}{\zeta^{d-1}} - \frac{1}{2}\zeta \right) \pa_\zeta   + \nu \pa_\zeta^2,  \qquad \quad \Pc = P_{1} \pa_\zeta + P_{0}, \label{def:AsPc}\\
P_{1} &= \frac{ Q_\nu - 1}{\zeta^{d-1}}  - \nu\frac{(d-1)}{\zeta}  + \frac{m_\ep}{\zeta^{d-1}} + \left(\frac{R_\tau}{R} + \frac{1}{2} \right) \zeta, \qquad P_{0}= \frac{\pa_\zeta Q_\nu}{\zeta^{d-1}} - \frac{M_\tau}{M}, \label{def:P1P0}
\end{align}
and the error $E$ is defined from \eqref{eq:Q}, 
\begin{align}
E&= \pa_\tau Q_\nu - \frac{M_\tau}{M} Q_\nu + \left[Q_\nu \left(\frac{1}{\zeta^{d-1}} -1\right)  - \frac{\zeta - 1}{2} -\nu\frac{(d-1)}{\zeta} + \left(\frac{R_\tau}{R} + \frac{1}{2} \right) \zeta \right]\pa_\zeta Q_\nu.\label{def:Eerror}
\end{align}
Equation \eqref{eq:mep} dampens derivatives in the sense that the equation for $m_{\ep, 1} = \pa_\zeta m_\ep$ is
\begin{equation}\label{eq:mep1}
\pa_\tau m_{\ep,1}= \As_1m_{\ep, 1} + \Pc_1 m_{\ep,1} + F \qquad \textup{for}\quad  \zeta \geq \zeta_+,
\end{equation}
where $\As_1$, $\Pc_1$ and $F$ are given by 
\begin{align}
\label{exterior:def:A1} \As_1  &=  - \left( \frac{d-1}{\zeta^{d}} + \frac{1}{2} \right) + \left(\frac{1}{\zeta^{d-1}} - \frac{1}{2}\zeta \right)\pa_\zeta  + \nu \pa_\zeta^2,\\
\label{exterior:def:P1} \Pc_1  &= P_1 \pa_\zeta + \big(\pa_\zeta P_{1} + P_0\big), \qquad F = \pa_\zeta E +\pa_\zeta P_0 m_\ep.
\end{align}
The damping of Equation \eqref{eq:mep1} is formalized using supersolutions. We introduce
\begin{equation} \label{exterior:def:phi1phi2}
\phi_1(\zeta, \tau) = \frac{1}{2} K^\frac{5}{4} e^{-\kappa \tau} e^{- \frac 38  \frac{\zeta -\zeta_+}{\nu}  }, \qquad \phi_2(\tau) = \frac{1}{2} e^{-\kappa \tau}\zeta^{d-1}.
\end{equation}

\begin{lemma} \label{lem:supersolutionmain}

Recall $ \hat \chi_{\eta}$ is defined by \eqref{def:hatchi_eta}. There exist $\eta^*(d)>0$ and $\kappa^*>0$, such that for any $0<\kappa\leq \kappa^*$ and $0<\eta\leq \eta^*$, for any $K,\bar M_0,A>0$, for $\tau_0$ large enough, one has for all $ \tau_0\leq \tau \leq \tau_1$ and $\zeta \geq  \zeta_+$,
\begin{equation}\label{est:Linearterms1}
\big( \pa_\tau - \As_1\big) \big(\phi_1 \hat \chi_{\eta}+ \phi_2 \big)(\zeta,\tau)  \geq  \frac{1}{16 \nu}\phi_1(\zeta, \tau) \hat \chi_{\eta} + \frac{3}{16}\phi_2(\tau).
\end{equation}

\end{lemma}

\begin{proof}

We first compute
\begin{align*}
(\pa_\tau -\As_1) (\phi_1 \hat \chi_{\eta}) =  \hat \chi_{\eta} \big( \pa_\tau \phi_1 - \As_1 \phi_1 \big) - [\As_1, \hat \chi_{\eta}] \phi_1,
\end{align*}
with the commutator 
$$
[\As_1, \hat \chi_{ \eta}] = 2\nu {\hat \chi}'_{ \eta} \pa_\zeta  + \left[\nu  {\hat \chi}''_{ \eta} + \left( \frac{1}{\zeta^{d-1}} - \frac{1}{2}\zeta \right)  {\hat \chi}'_{ \eta} \right] .
$$
Recall $\zeta_+=1+4\nu |\log \nu|$. We compute using \eqref{exterior:def:phi1phi2} and \eqref{exterior:def:A1}:
\begin{align*}
\frac{\pa_\tau \phi_1 - \As_1 \phi_1}{\phi_1} &= \frac{3}{8\nu}\left[ \frac{1}{\zeta^{d-1}} - \frac{1}{2}\zeta - \frac 38 + \frac{\nu_\tau}{\nu} \big(\zeta - 1 - 4\nu \big) \right]+ \frac{d-1}{\zeta^d} + \frac{1}{2} - \kappa.
\end{align*}
Since $\frac{\nu_\tau}{\nu} = -\frac{d-2}{2}+o(1)$ (a consequence of \eqref{eq:Modnu}) where the $o(1)$ is as $\tau_0\to \infty$, there is a constant $0< \eta \ll 1$ such that for $\tau_0$ large enough
$$ \frac{1}{\zeta^{d-1}} - \frac{1}{2}\zeta - \frac 38 + \frac{\nu_\tau}{\nu} \big(\zeta - 1 - 4\nu \big) \geq \frac{1}{16} , \quad \textup{for}\;\; \zeta \in[1, 1+2\eta]. $$
We also have for $\kappa<1/2$, using again $\frac{\nu_\tau}{\nu} = -\frac{d-2}{2}+o(1)$, $\frac{d-1}{\zeta^d} + \frac{1}{2} - \kappa > 0$ for $\zeta > 0$. Hence, combining the three above equality and inequalities we end up with 
\be \label{supersolution:inter1}
\pa_\tau \phi_1 - \As_1 \phi_1\geq \frac{1}{16\nu} \phi_1 \quad \textup{for}\;\; \zeta \in [1, 1 + 2\eta].
\ee
Using that the support of $\hat \chi_{\tilde{ \eta}}'$ and $\hat \chi_{\tilde{ \eta}}''$ is $1+\eta \leq \zeta  \leq 1+ 2 \eta$, \eqref{exterior:def:phi1phi2}, $\zeta_+=1+4\nu |\log \nu|$, and that for $\zeta\geq 1+\eta$ there holds $e^{-\frac 38 \frac{\zeta-\zeta_+}{\nu}}\leq \nu^{-\frac 32}e^{-\frac{3 \eta}{8\nu}}$ we estimate 
\be \label{supersolution:inter2}
\Big| [\As_1, \hat \chi_{\eta}] \phi_1 \Big| \lesssim K^{\frac{5}{4}}  e^{-\kappa \tau} \nu^{-\frac 32} e^{- \frac{3\eta}{8\nu}} \leq \nu \phi_2,
\ee
for $\tau_0$ large enough depending on $\eta,K$. A direct computation using \eqref{exterior:def:phi1phi2} yields for $\kappa<\frac 14$:
\begin{equation}\label{eq:varphi2}
\pa_\tau \phi_2 - \As_1 \phi_2 = \left[ - \kappa + \frac{1}{2}+\frac{d-1}{2\zeta}-\nu \frac{(d-1)(d-2)}{\zeta^2} \right] \phi_2 \geq \frac{1}{4} \phi_2 \quad \textup{for} \;\; \zeta > 0.
\end{equation}
Combining \eqref{supersolution:inter1}, \eqref{supersolution:inter2} and \eqref{eq:varphi2} yields the desired estimate \eqref{est:Linearterms1} for $\tau_0$ large enough.
\end{proof}

\begin{lemma} \label{lem:rightboundary}

There exist $\kappa^*(d)>0$ and $A^*>0$, such that for any $0<\kappa\leq\kappa^*$, $A\geq A^*$ and $K>0$ with $e^{3A/10}\leq K\leq e^{\frac 32 A}$, for any $\bar M_0,\eta>0$, for $\tau_0$ large enough one has for all $s\geq s_0$:
\begin{equation} \label{bd:boundaryimproved}
\big|\pa_\xi m_q(s, \xi_-)\big|+\big|\pa_\xi m_q(s, \xi_+)\big| \lesssim  K \nu e^{-\kappa \tau}+\nu^2.
\end{equation}

\end{lemma}

\begin{proof}

We only establish the estimate \eqref{bd:boundaryimproved} at $\xi=\xi_+$ since those estimate at $\xi=\xi_-$ can be obtained by a very similar computation. We use a standard parabolic regularization argument. We write $\chi=\chi_{1,\xi_+}$ to ease notations. Note $\textup{supp} (\chi_{1,\xi_+})\subset \{ | \xi-\xi_+|\leq 2\}$ and ${\bf 1}_{ \{ | \xi-\xi_+|\leq 2\}}\lesssim \chi_{2,\xi_+}$. We introduce $\tilde m_q=\chi m_q$ which solves from \eqref{eq:mqxis}:
\begin{align}
\label{boundary:id:evo} &\pa_s \tilde m_q=\pa_{\xi}^2 \tilde m_q+f, \qquad f=\tilde f+\chi \Psi, \\
\nonumber &\tilde f=\chi\left((Q-\frac 12)\pa_\xi m_q+\pa_\xi Q m_q+L(m_q)+\frac{m_q\pa_\xi m_q}{(1+\nu \xi)^{d-1}}\right)+(\pa_s \chi-\pa_\xi^2\chi)m_q-2\pa_\xi \chi \pa_\xi m_q.
\end{align}
We now estimate $f$. Using \eqref{mod:bd:pointwiseL}, \eqref{eq:ModR}, \eqref{eq:ModM} and \eqref{bd:pointwisemq}, we get $|\tilde f|\lesssim (|m_q|+|\pa_\xi m_q|)\chi_{2,\xi_+}$. Thus, applying Cauchy-Schwarz, then using \eqref{est:mqL2ome} and $\omega_0(\xi)\approx \nu^{-2}$ on $\textup{supp} (\chi_{2,\xi_+})$ yields
\be  \label{boundary:bd:inter3}
\| \tilde f\|_{L^2(\mathbb R)}\lesssim \| m_q\|_{\inn}\|\chi_{2,\xi_+} \sqrt{\omega_0}^{-1}\|_{L^2}\lesssim \nu K e^{-\kappa \tau}.
\ee
Next, using \eqref{eq:ModM} and \eqref{bd:boundarynorm} and $e^{A/10}\leq K$ yields $|M_s/M|\lesssim \nu^2+\nu K^{5/4}e^{-3A/8}e^{-\kappa \tau}$. Hence, we estimate from \eqref{mod:bd:pointwisePsi}, $Q(\xi)=1+O(e^{-|\xi|/2})$ and $|\pa_\xi Q(\xi)|\lesssim e^{-|\xi|/2}$, for all $|\xi-\xi_+|\leq 2$:
$$
|\Psi (s,\xi)|\lesssim |\frac{M_s}{M}|Q(\xi)+\left(|\frac{R_\tau}{R}+\frac 12|+|\frac{M_s}{M}||\log \nu|+\nu|\log \nu| \right)|\pa_\xi Q(\xi)|\lesssim \nu^2+\nu K^{5/4}e^{-3A/8}e^{-\kappa \tau}
$$
and hence,
\be \label{boundary:bd:inter4}
\| \chi \Psi \|_{L^2(\mathbb R)}\lesssim  \nu^2+\nu K^{5/4}e^{-3A/8}e^{-\kappa \tau}.
\ee
Injecting \eqref{boundary:bd:inter3} and \eqref{boundary:bd:inter4} in \eqref{boundary:id:evo} yields
\be \label{boundary:bd:inter2}
\| f(s) \|_{L^2(\mathbb R)}\lesssim  \nu K e^{-\kappa \tau}(1+K^{1/4}e^{-3A/8})+\nu^2.
\ee
For $s\geq s_0$, we introduce $\tilde s_0=\max(s-1,s_0)$ and get from \eqref{boundary:id:evo} the representation formula
\be \label{boundary:id:rep}
\tilde m_q=\underbrace{K_{s-\tilde s_0}*\tilde m_q(\tilde s_0)}_{=\tilde m_q^1}+\underbrace{\int_{\tilde s_0}^s K_{s-s'}*f(s')ds'}_{=\tilde m_q^2}, \qquad K_s(\xi)=(4\pi s)^{-\frac 12}e^{-\frac{\xi^2}{4s}}.
\ee
Note that $\| K_s\|_{L^1}=1$ and $\| \pa_\xi K_s\|_{L^2}\lesssim s^{-3/4}$ by direct computations. Hence, if $\tilde s_0=s_0$ then by Young's inequality, the localization of $\chi$, \eqref{bd:bootstrap:initialregularity} and $\pa_\xi=\nu\pa_\zeta$, \eqref{bd:bootstrap:initial} and \eqref{iq:LinfSobolev},
\begin{align} 
\nonumber & \|\pa_\xi \tilde m_q^1\|_{L^\infty(\mathbb R)} \ \lesssim  \  \| K_{s-s_0}\|_{L^1(\mathbb R)} \|\pa_\xi \tilde m_q(s_0)\|_{L^\infty (\mathbb R)} \\
\label{boundary:bd:inter5} &\qquad \quad \lesssim \|\pa_\xi m_q(s_0)\|_{L^\infty ( |\xi-\xi_+(s_0)|\leq 2 )}+\| m_q(s_0)\|_{L^\infty ( |\xi-\xi_+(s_0)|\leq 2 )} \  \lesssim \ \nu e^{-\kappa \tau},
\end{align}
while if $\tilde s_0=s-1$, then using \eqref{bd:pointwisemq} and the localization of $\chi$ yields
\be \label{boundary:bd:inter6}
\|\pa_\xi \tilde m_q^1\|_{L^\infty(\mathbb R)}\lesssim \|\pa_\xi K_1\|_{L^2(\mathbb R)}\| \tilde m_q(s-1)\|_{L^2(\mathbb R)}\lesssim \nu K e^{-\kappa \tau}.
\ee
Finally, using \eqref{boundary:bd:inter2} and $\int_{0}^1 s^{-3/4}ds<\infty$ we obtain
\be \label{boundary:bd:inter7}
\| \pa_\xi \tilde m_q^2\|_{L^\infty}\lesssim \int_{\tilde s_0}^s \| \pa_\xi K_{s-s'}\|_{L^2}\| f\|_{L^2}ds'\lesssim  \nu K e^{-\kappa \tau}(1+K^{1/4}e^{-3A/8})+\nu^2.
\ee
Injecting \eqref{boundary:bd:inter5}, \eqref{boundary:bd:inter6} and \eqref{boundary:bd:inter7} in \eqref{boundary:id:rep} and $K\leq e^{3 A/2}$ yields the estimate \eqref{bd:boundaryimproved}.
\end{proof}

\begin{lemma} \label{lemm:mep_mid} There exists $K^*\geq 1$ and $\kappa^*>0$ such that if $A,K,\kappa,\eta,\bar M_0$ satisfy the conditions of Lemmas \ref{lem:supersolutionmain} and \ref{lem:rightboundary}, with $K\geq K^*$ and $0<\kappa\leq \kappa^*$, then for all $\tau_0\leq \tau \leq \tau_1$:
\begin{equation}\label{est:mep_mid}
\big|\pa_\zeta m_\ep (\zeta, \tau)\big| \leq \phi_1(\zeta, \tau)  \hat \chi_{\eta}+  \phi_2(\tau) \qquad  \textup{for}\;\; \zeta \geq \zeta_+,
\end{equation}
where $\phi_1$ and $\phi_2$ are defined in \eqref{exterior:def:phi1phi2}, and $\hat \chi_{\eta}$ is introduced in \eqref{def:hatchi_eta}. 
\end{lemma}
\begin{proof}

\textbf{Step 1}. \emph{Proof assuming a technical estimate.} The proof relies on the standard parabolic comparison principle, where we shall construct a super/sub solution for the equation satisfied by $\pa_\zeta m_\ep$. We claim the following: for $\tau_0$ large enough, for all $\tau\geq \tau_0$ and $\zeta \geq \zeta_+$,
\begin{equation}\label{est:remainder_mid}
\left| \Pc_1 \big(\phi_1(\zeta, \tau) \hat \chi_{\eta}(\zeta) + \phi_2(\tau)  \big) \right|  +|F(\zeta, \tau)| \leq  \frac{1}{32\nu}\phi_1(\zeta, \tau) \hat \chi_{\eta} + \frac{1}{8}\phi_2(\tau).
\end{equation}
We proceed with the proof of \eqref{est:mep_mid}, establishing \eqref{est:remainder_mid} later on. From \eqref{est:Linearterms1} and \eqref{est:remainder_mid}, we obtain that $\phi_1 \hat \chi_{\eta}+ \phi_2$ is a supersolution to \eqref{eq:mep1} for $\tau \geq \tau_0 $ and $\zeta \geq \zeta_+ $ thanks to \begin{equation} \label{supersolution:bd:inter1}
\big(\pa_\tau  - \As_1 - \Pc_1\big) \big(\phi_1(\zeta, \tau) \hat \chi_{\eta}+ \phi_2(\tau) \big)  - F > 0 .
\end{equation}
Next, at the initial time $\tau_0$ we have because of \eqref{bd:bootstrap:initialregularity} that for all $\zeta\geq \zeta_+(\tau_0)$,
\be \label{supersolution:bd:inter2}
m_{\varepsilon,1}(\tau_0,\zeta)\leq  \phi_2(\tau_0)\leq (\phi_1\hat \chi_{\eta}+ \phi_2)( \tau_0,\zeta).
\ee
At the boundary, we combine \eqref{bd:boundaryimproved}, \eqref{bd:estimationtau} and \eqref{exterior:def:phi1phi2} to get for all $\tau_0\leq \tau \leq \tau_1$:
\be \label{supersolution:bd:inter3}
m_{\varepsilon,1}(\tau,\zeta_+) \leq \phi_1( \tau,\zeta_+) \leq (\phi_1\hat \chi_{\eta}+ \phi_2) (\tau,\zeta_+),
\ee
if $\kappa$ is small enough, $K$ is large enough and then $\tau_0$ is large enough. Combining \eqref{eq:mep1}, \eqref{supersolution:bd:inter1}, \eqref{supersolution:bd:inter2} and \eqref{supersolution:bd:inter3}, we can apply the maximum principle for $\phi_1\hat \chi_{\eta}+ \phi_2-m_{\varepsilon,1}$ as a supersolution for the parabolic operator $\pa_\tau  - \As_1 - \Pc_1$ on the set $\{\tau_0\leq \tau \leq \tau_1, \ \zeta\geq \zeta_+\}$ that is nonnegative at its boundary, and we obtain $\phi_1 \hat \chi_{\eta}+ \phi_2-m_{\varepsilon,1}\geq 0$ on this set, i.e.
$$
m_{\varepsilon,1}(\zeta,\tau)\leq \phi_1(\zeta, \tau) \hat \chi_{\eta}(\zeta)+ \phi_2(\tau).
$$
The bound $-m_{\varepsilon,1}\leq \phi_1 \hat \chi_{\eta}+ \phi_2$ is obtained similarly. Combining these two bounds concludes the proof of \eqref{est:mep_mid}.\\
\noindent \textbf{Step 2}. \emph{Control of the lower order terms.} We now prove \eqref{est:remainder_mid}. Recall \eqref{def:P1P0} and \eqref{def:AsPc}. From \eqref{est:mqbydev} and \eqref{est:mepLinf_out}, we write for $\zeta  \geq \zeta_+$,
\begin{align}
\frac{|m_\ep(\zeta, \tau)|}{\zeta^{d-1}} &\leq \frac{1}{\zeta^{d-1}} \int_{1 + \nu \xi^*}^\zeta |\pa_\zeta m_\ep(\zeta')| d\zeta'  \leq \zeta K^{\frac 54}e^{-\kappa \tau}.\label{est:mep1}
\end{align}
Recall $\zeta\geq \zeta_+=1+4\nu |\log \nu|$ corresponds to $\xi \geq \xi_+=4|\log \nu|$. We  use the exponential decay $Q(\xi) = 1+\Oc(e^{-\xi/2})$ for $\xi \geq 0$ with $e^{-\xi/2}\lesssim \nu^{2}$ for $\zeta\geq \zeta_+$, \eqref{est:mep1}, and \eqref{eq:ModR} to estimate  for $ \zeta  \geq \zeta_+$, 
\begin{align*}
 |P_1(\zeta, \tau)| &\lesssim |Q_\nu - 1| +\frac \nu \zeta+ \frac{|m_\ep (\zeta, \tau)|}{\zeta^{d-1}} + \left|\frac{R_\tau}{R} + \frac{1}{2} \right|\zeta \\
& \lesssim e^{-\frac \xi2} + \nu   + \| \pa_\zeta m_\ep(\tau)\|_{L^\infty(\zeta\geq \zeta_+)} + \left( \nu+e^{-\frac A4}\|m_q(\tau)\|_\inn + A \| m_\varepsilon \|_{\textup{bou}}\right)\zeta \lesssim  K^\frac{5}{4} e^{-\kappa \tau}\zeta ,
\end{align*}  
where we used \eqref{est:mqL2ome} and \eqref{bd:boundarynorm} for the last inequality and took $\kappa $ small enough $\tau_0$ large enough, and similarly
\begin{align*}
|\pa_\zeta P_1(\zeta, \tau)| &\lesssim |\pa_\zeta Q_\nu (\zeta)|+|Q_\nu - 1|  + \nu + \frac{|m_\ep(\zeta, \tau)|}{\zeta^d}+ \frac{|\pa_\zeta m_\ep(\zeta, \tau)|}{\zeta^{d-1}} +\left|\frac{R_\tau}{R} + \frac{1}{2} \right|\lesssim K^\frac{5}4e^{-\kappa \tau},
\end{align*}
and, using in addition \eqref{eq:ModM} and $M_\tau=\nu^{-1}M_s$,
\begin{align*}
|P_0(\zeta, \tau)| \lesssim |\pa_\zeta Q_\nu (\zeta)| + \left|\frac{M_\tau}{M} \right| \lesssim \nu + \|m_q(\tau)\|_\inn+\nu^{-1}\| m_\ep(\tau)\|_{\textup{bou}} \lesssim K^\frac{5}4e^{-\kappa \tau}.
\end{align*}
Hence, using that $\phi_1\leq K^{5/4}\nu^{-3/2} e^{-\frac{3\eta}{8\nu}}$ for $\zeta\geq 1+\eta$ and $\phi_2=e^{-\kappa \tau}\zeta^{d-1}/2$:
\begin{align}
\nonumber & \big|\Pc_1 \phi_1(\zeta, \tau) \hat \chi_{\eta} \big| +\big|\Pc_1 \phi_2(\tau) \big|  \\
\nonumber &  \leq \ |P_1 \pa_\zeta \phi_1(\zeta, \tau)| \hat \chi_{\eta} + |\pa_\zeta P_1 + P_0 | \phi_1 \hat \chi_{\eta}+ |P_1 \pa_\zeta \hat \chi_{\eta}| \phi_1+\big( |\pa_\zeta P_1 + P_0 | \phi_2(\tau)+|P_1||\pa_\zeta \phi_2|\big) \\
\nonumber &  \lesssim \left( \frac{3}{8 \nu} |P_1| +  |\pa_\zeta P_1| + |P_0 | \right) \phi_1 \hat \chi_{\eta}   +|P_1|\phi_1 \mathbf{1}_{\{1+\eta\leq \zeta\leq 1+2\eta\}}+\big( |\pa_\zeta P_1| + |P_0 |+\zeta^{-1}|P_1|\big) \phi_2(\tau) \\
 & \lesssim \frac{3}{8\nu} K^\frac{5}{4} e^{-\kappa \tau} \phi_1 \hat \chi_{\eta}  + K^\frac{5}{2} \nu^{-\frac 32} e^{- \frac{3\eta}{8\nu}} \phi_2  +  K^{\frac{5}{4}} e^{-\kappa \tau} \phi_2(\tau) \leq \frac{1}{64\nu}\phi_1(\zeta, \tau) \hat \chi_{\eta} + \frac{1}{16}\phi_2(\tau), \label{exterior:bd:inter3}
\end{align}
for $\tau_0$ large enough.

We now estimate the source term $F= \pa_\zeta E +\pa_\zeta P_0 m_\ep$. Using \eqref{def:P1P0}, $|\pa_\xi^j Q(\xi)|\lesssim e^{-|\xi|/2}$ for $j=1,2$ and \eqref{est:mep1} we obtain for $\zeta  \geq \zeta_+ $,
\begin{align}
\nonumber & \big|\pa_\zeta P_0(\zeta) m_\ep(\zeta)\big| \ = \ \big|\pa_\zeta \big(\frac{\pa_\zeta Q_\nu}{\zeta^{d-1}} \big)(\zeta ) m_\ep(\zeta)\big|  \ \lesssim  \ \Big(|\pa_\zeta^2 Q_\nu(\zeta)| + |\pa_\zeta Q_\nu(\zeta)|\Big) \frac{|m_\ep(\zeta)|}{\zeta^{d-1}}\\
\label{exterior:bd:inter2} &\qquad \qquad \qquad \qquad  \lesssim \frac{1}{\nu^2} e^{-\frac{\xi}{2}}  \| \pa_\zeta m_\ep\|_{L^\infty(\zeta \geq \zeta_+)}.
\end{align}
Next, using $|\nu_\tau/\nu |\lesssim 1$ from \eqref{eq:Modnu} and $|\pa_\xi^j Q(\xi)|\lesssim e^{-|\xi|/2}$ for $j=1,2$, for $\zeta \geq \zeta_+$,
\begin{align*}
\big| \pa_\zeta \pa_\tau  Q_\nu (\zeta)\big| =\frac{1}{\nu}|\frac{\nu_\tau}{\nu}| \big| \pa_\xi Q(\xi) +\xi \pa_\xi^2 Q(\xi)\big|   \lesssim  \frac{1}{\nu} \xi e^{-\frac{\xi}{2}},
\end{align*}
and similarly, using that $\zeta-1\geq \nu $ yields the estimate
\begin{align*}
& \left| \pa_\zeta \left(  Q_\nu \pa_\zeta  Q_\nu \Big(\frac{1}{\zeta^{d-1}} - 1 \Big)  \right) \right| + \Big| \pa_\zeta \big[(\zeta  -1) \pa_\zeta Q_\nu\big] \Big| + \Big| \pa_\zeta \big[ \frac{\nu}{\zeta} \pa_\zeta  Q_\nu\big] \Big|\\
&\qquad \qquad \qquad \lesssim \left(\big| \pa_\zeta  Q_\nu \big|^2 + \big| \pa_\zeta^2  Q_\nu\big|\right) (\zeta -1) + |\pa_\zeta  Q_\nu| \ \lesssim \ \frac{1}{\nu}\xi e^{-\frac{\xi}{2}}.
\end{align*} 
We estimate using $M_\tau=\nu^{-1}M_s$, \eqref{eq:ModR} and \eqref{eq:ModM} for $\zeta \geq \zeta_+$:
\begin{align*}
\left|\left(\frac{R_\tau}{R} + \frac{1}{2} \right) \pa_\zeta (\zeta \pa_\zeta  Q_\nu) - \frac{M_\tau}{M}\pa_\zeta  Q_\nu \right| & \lesssim  \left|\frac{R_\tau}{R} + \frac{1}{2} \right| |\zeta \pa_\zeta^2  Q_\nu| + \left( \left|\frac{R_\tau}{R} + \frac{1}{2} \right| + \left| \frac{M_\tau}{M} \right| \right) |\pa_\zeta  Q_\nu|\\
&\qquad \lesssim  |\zeta \pa_\zeta^2  Q_\nu| + |\pa_\zeta  Q_\nu|  \ \lesssim  \ \frac{1}{\nu^2} \xi e^{-\frac \xi 2},
\end{align*}
where we used the rough estimate $\zeta\leq \xi$. Injecting the three above inequalities in \eqref{def:Eerror} shows $|\pa_\zeta E|\lesssim \nu^{-2}\xi e^{-\xi/2}$. Combining this with \eqref{exterior:bd:inter2} gives $|F|\lesssim  \nu^{-2}\xi e^{-\xi/2}$. Now, observe from the definition \eqref{exterior:def:phi1phi2} of $\phi_1$ and $\phi_2$ that $\xi e^{-\xi/2}\leq \nu^2|\log \nu| e^{\kappa \tau}\phi_1(\tau,\zeta)$ for $\zeta \geq \zeta_+=1+4\nu |\log \nu|$ for $K$ large enough, and $\xi e^{-\xi/2}\leq e^{-\eta/3\nu}\phi_2(\tau)$ for $\zeta \geq 1+\eta$ for $\tau_0$ large enough depending on $\eta$. Hence, for $ \zeta\geq \zeta_+$,
\begin{align*}
|F|\lesssim  \nu^{-2}\xi e^{-\xi/2}\lesssim  |\log \nu| e^{\kappa \tau}\phi_1(\tau,\zeta)\hat \chi_\eta+\nu^{-2}e^{-\frac{\eta}{3\nu}}\phi_2(\tau)\leq  \frac{1}{64\nu}\phi_1(\zeta, \tau) \hat \chi_{\eta} + \frac{1}{16}\phi_2(\tau),
\end{align*}
where we used \eqref{bd:estimationtau} and took $\kappa$ small enough, and then $\tau_0$ large enough. Combining the above inequality and \eqref{exterior:bd:inter3} shows the desired estimate \eqref{est:remainder_mid}.
\end{proof}

We now turn to the control of the solution over the interval $\zeta \in (0, \zeta_-)$. We have by \eqref{eq:phisfrsft},
\begin{equation}\label{eq:mepleft}
\pa_\tau m_\ep = \As^- m_\ep + \Pc^- m_\ep + E \qquad \textup{for}\quad  \zeta \leq \zeta_-,
\end{equation}  
where
\begin{align}
& E= \pa_\tau \bar Q_\nu + \left[  \bar Q_\nu \left(\frac{1}{\zeta^{d-1}} -1\right)  - \frac{\zeta - 1}{2}  -\nu\frac{(d-1)}{\zeta}  +  \left(\frac{R_\tau}{R} + \frac{1}{2} \right) \zeta \right] \pa_\zeta \bar Q_\nu  - \frac{M_\tau}{M} \bar Q_\nu.\label{def:Eerrorleft}\\
& \As^- = - \frac{1}{2}\zeta  \pa_\zeta   + \nu \left(\pa_\zeta^2-\frac{d-1}{\zeta}\pa_\zeta\right), \label{def:AsPcleft}\\
& \Pc^- = P_{1}^- \pa_\zeta + P_{0}^-, \quad P_{1}^- = \frac{ \bar Q_\nu }{\zeta^{d-1}}  + \frac{m_\ep}{\zeta^{d-1}} + \left(\frac{R_\tau}{R} + \frac{1}{2} \right) \zeta, \quad P_{0}^-= \frac{\pa_\zeta \bar Q_\nu}{\zeta^{d-1}} - \frac{M_\tau}{M}. \label{def:P1P0left}
\end{align}
The equation for $m_{\ep, 1} = \pa_\zeta m_\ep$ reads as
\begin{equation}\label{eq:mep1left}
\pa_\tau m_{\ep,1}= \As_1^- m_{\ep, 1} + \Pc_1^- m_{\ep,1} + F^- \qquad \textup{for}\quad  \zeta \leq \zeta_-,
\end{equation}
where
\begin{align}
\label{exterior:def:A1left} \As_1^-  &= -\frac{1}{2}\zeta \pa_\zeta - \frac{1}{2} +\nu \left(\pa_\zeta^2-\frac{d-1}{\zeta}\pa_\zeta + \frac{d-1}{\zeta^2}\right), \\
\label{exterior:def:P1left} \Pc_1^-  &= P_1^- \pa_\zeta + \big(\pa_\zeta P_{1}^- + P_0^-\big), \qquad F^- = \pa_\zeta E +\pa_\zeta P_0^- m_\ep.
\end{align}
We introduce
\begin{equation} \label{exterior:def:phi1phi2left}
\phi_1^-(\zeta, \tau) = \frac{1}{2} K^\frac{5}{4} e^{-\kappa \tau} e^{- \frac 38  \frac{\zeta_--\zeta}{\nu}  }, \qquad \phi_2^-(\zeta,\tau) = \frac{1}{2} \nu \zeta^{d-1} e^{-\kappa \tau}.
\end{equation}

\begin{lemma} \label{lem:supersolutionmainleft}

There exist $\eta^*(d)>0$ and $\kappa^*>0$, such that for any $0<\kappa\leq \kappa^*$ and $0<\eta\leq \eta^*$, for any $K,\bar M_0,A>0$, for $\tau_0$ large enough, one has for all $ \tau_0\leq \tau \leq \tau_1$ and $\zeta \leq  \zeta_-$,
\begin{equation}\label{est:Linearterms1}
\big( \pa_\tau - \As_1^-\big) \big(\phi_1^- \hat \chi_{\eta}+ \phi_2^- \big)(\zeta,\tau)  \geq  \frac{1}{16 \nu}\phi_1^-(\zeta, \tau) \hat \chi_{\eta} + \frac 14 \phi_2^-(\tau).
\end{equation}

\end{lemma}

\begin{proof}

By a direct computation, one obtains for all $ 1-2\eta \leq \zeta \leq 1$,
$$
\frac{\pa_\tau \phi_1^--\As_1^- \phi_1^-}{\phi_1^-} = \frac{3}{8\nu} \left( \frac 12-\frac 38+\frac{\nu_\tau}{\nu}(1-4\nu-\zeta)\right)+\frac 12 +\frac{3}{8\zeta}(d-1)-\kappa-\nu\frac{d-1}{\zeta^2} \geq \frac{1}{16 \nu},
$$
where we used the fact that $0<\nu \to 0$ uniformly as $\tau_0\to \infty$, $|\frac{\nu_\tau}{\nu}|\lesssim 1 $ from \eqref{eq:Modnu}, and took $\eta>0$ small enough. Using the fact that $\Big(\pa_\zeta^2 - \frac{d - 1}{\zeta} \pa_\zeta + \frac{d-1}{\zeta^2}   \Big) \zeta^{d-1} = 0$, we obtain
\be \label{bdbdbd}
\frac{\pa_\tau \phi_2^--\As_1^- \phi_2^-}{\phi_2^-}=\frac d2-\kappa+\frac{\nu_\tau}{\nu}=1-\kappa+o(1)\geq \frac 12,
\ee
where we used that $|\frac{\nu_\tau}{\nu}=\frac{d-2}{2}+o(1) $ from \eqref{eq:Modnu}. With a computation that is so similar to that establishing \eqref{supersolution:inter2} in the proof of Lemma \ref{lem:supersolutionmain}, so that we omit it, we moreover have for $\zeta \leq \zeta_-$,
$$
|[\As_1^- ,\hat \chi_\eta]\phi_1^-| =o(\phi_2^-),
$$
where $[\As_1^- ,\hat \chi_\eta]=\As_1^- \hat \chi_\eta-\hat \chi_\eta \As_1^-$ and $o()$ stands for $\tau_0\to \infty$ and is uniform in $\tau,\zeta$. Combining the three above inequalities yields the desired estimate \eqref{est:Linearterms1}.
\end{proof}

\begin{lemma} \label{lemm:mep_out} We assume that  $m_q(\tau) \in \Sc_{K, \kappa}(\tau)$ for $\tau \in [\tau_0 , \tau_1]$ and $\tau_1 > \tau_0 \gg 1$, there exists $0 < \eta  \ll 1$ and the following holds true for all $\tau \in [\tau_0, \tau_1]$:
\begin{equation}\label{est:mep_out}
\big|\pa_\zeta m_\ep (\zeta, \tau)\big| \leq \phi_1^-(\zeta, \tau)  \hat \chi_{\eta}+ \phi_2^-(\zeta,\tau) \quad  \textup{for}\;\; 0 < \zeta \leq 1 - 4\nu |\log \nu|,
\end{equation}
where $\hat \chi_{\eta}$ is introduced in \eqref{def:hatchi_eta}. 
\end{lemma}

\begin{proof} The proof is the same as for \eqref{est:mep_mid} by using the comparison principle, so we skip redundant details. Due to the localized cut-off function $\hat \chi_{\eta}$, we note that the estimate \eqref{est:remainder_mid} hold true also for $\zeta \in [1 - 2\eta,\zeta_-] $, so that using in addition \eqref{est:Linearterms1} on this interval we have
\begin{equation}\label{est:bdbdbdbdbd} \big[\pa_\tau - \As_1^- - \Pc_1^-\big] \big( \phi_1^-(\zeta, \tau) \hat \chi_{\eta} + \phi_2^-(\tau) \big)  - F^-(\zeta, \tau)  \geq \frac{1}{32 \nu} \phi_1(\zeta, \tau) \hat \chi_{\eta} + \frac{1}{8}\phi_2(\tau).
\ee
Since $\hat \chi_\eta\equiv 0$ for $\zeta\leq 1-2\eta$ it remains to check that $\phi_2^-$ satisfies
\begin{equation}\label{est:phi3}
\big[\pa_\tau - \As_1^- - \Pc_1^-\big] \phi_2^-(\zeta, \tau)  - F^-(\zeta, \tau) > 0, \quad \textup{for}\;\; \zeta \in (0, 1 - 2 \eta). 
\end{equation}
To this end, we first recall from \eqref{bdbdbd} the estimate $
\pa_\tau \phi_2^-- \As_1^-\phi_2^- \geq \frac 12 \phi_2^- $. We estimate for $\zeta \in (0, 1 - 2\eta)$ by using \eqref{est:mepLinf_out2} and $m_\varepsilon (0)=0$, 
$$|m_\ep(\zeta, \tau)| = \left|\int_0^\zeta m_{\ep, 1}(\zeta', \tau) d\zeta' \right| \lesssim \nu e^{-\kappa \tau} \zeta^d.$$
 We also have by the definition \eqref{def:Qbar} of $\bar Q_\nu$ and \eqref{eq:ModR}, 
\begin{align*}
 | P_1^- \zeta^{-1}|  + |\pa_\zeta P_1^-| + |P_0|&\lesssim \frac{|\pa_\zeta \bar Q_\nu|}{\zeta^{d-1}}+ \frac{|\bar Q_\nu|}{\zeta^{d}} + \frac{|m_\ep(\zeta, \tau)|}{\zeta^d}  + \frac{|m_{\ep,1}|}{\zeta^{d-1}}+ \left| \frac{R_\tau}{R} + \frac{1}{2} \right|\\
 & \lesssim \frac{1}{\nu \zeta_0^d} e^{-\frac{|\zeta - 1| }{2\nu}}  +  \nu e^{-\kappa \tau}+K^\frac 54 e^{-\kappa \tau} \lesssim \frac{\nu}{\zeta_0^d} + K^{\frac 54}e^{-\kappa \tau} \lesssim K^\frac 54 e^{-\kappa \tau}.
\end{align*}
Hence, for $\tau_0$ large enough, we have
$$
|\Pc_1^- \phi_2^-| \leq \Big((d-1)| P_1^- \zeta^{-1}| + |\pa_\zeta  P_1^-| + |P_0^-| \Big)\phi_2^- \leq CK^\frac 54 e^{-\kappa \tau }\phi_2^- \leq \frac{1}{16} \phi_2^-.
$$
For the estimate of $F^-$, we have the rough estimate for $\zeta \in (0, 1 - 2\eta)$,
\begin{align*}
 |F(\zeta, \tau)| \lesssim \frac{1}{\nu^2 \zeta^d}  e^{-\frac{|\xi|}{2}} \mathbf{1}_{\{\zeta \geq \zeta_0\}} \lesssim \frac{1}{\nu^2 \zeta_0^{2d-1}} e^{-\frac{\eta}{\nu}} e^{\kappa \tau} \phi_2^-  \lesssim \nu \phi_2^-
\end{align*}
for $\tau_0$ large enough. Gathering all the above estimates yields the estimate \eqref{est:phi3}. Combining \eqref{est:bdbdbdbdbd} and \eqref{est:phi3} shows that for all $0<\zeta\leq \zeta_-$:
$$
 \big[\pa_\tau - \As_1^- - \Pc_1^-\big] \big( \phi_1^-(\zeta, \tau) \hat \chi_{\eta} + \phi_2^-(\tau) \big)  - F^-(\zeta, \tau)  \geq \frac{1}{32 \nu} \phi_1(\zeta, \tau) \hat \chi_{\eta} + \frac{1}{8}\phi_2(\tau).
$$
The end of the proof of Lemma \ref{lemm:mep_out} is then exactly as that of Lemma \ref{lemm:mep_mid}, relying on the above inequality, so we omit it. \end{proof}

\subsection{Proof of  Proposition \ref{prop:trap_mq} and conclusion of the main theorem} \label{sec:4.4}
\begin{proof}[Proof of Proposition \ref{prop:trap_mq}] We first improve estimates introduced in Definition \ref{def:bootstrap} by a $\frac 12 $ factor. We claim that for all $\tau \in [\tau_0, \tau_1]$:
\begin{align}
\label{proofbt:1}& \frac{1}{2} e^{-\frac{\tau}{2}}\leq R(\tau) \leq 2e^{-\frac{\tau}{2}}, \qquad \frac{M_0}{2}\leq M(\tau) \leq 2M_0,\\
\label{proofbt:2} & \big| \pa_\zeta m_\ep(\zeta, \tau)\big| \leq \frac{1}2 e^{-\kappa \tau} \left(K^{\frac{5}{4}} e^{-\frac 38 \frac{\zeta-\zeta_+}{\nu} }  \hat \chi_{\eta} + 1 \right), \qquad \textup{for}\; \zeta \geq \zeta_+, \\
\label{proofbt:3}& \big| \pa_\zeta m_\ep(\zeta, \tau)\big| \leq \frac{1}2 e^{-\kappa \tau} \left(K^\frac{5}{4} e^{-\frac{3}{8} \frac{ \zeta_--\zeta}{\nu} } \hat \chi_{\eta} +\nu \zeta^{d-1}  \right)  \qquad \textup{for}\; 0 < \zeta \leq \zeta_-,\\
\label{proofbt:4} & \|m_q(\tau) \|_\inn \leq \frac K2 e^{-\kappa \tau}.
\end{align}
The inequality \eqref{proofbt:1} is proved in Corollary \ref{cor:mod}. The inequalities \eqref{proofbt:2} and \eqref{proofbt:3} are proved in Lemmas \ref{lemm:mep_mid} and \ref{lemm:mep_out} (using \eqref{exterior:def:phi1phi2}). Hence it only remains to prove \eqref{proofbt:4}. Let $f(\tau) =\| m_q(\tau)\|_\inn^2$, we aim at proving
\begin{equation}\label{est:ftaus}
f(\tau) \leq \frac{K^2}{4} e^{-2\kappa \tau}, \quad \forall \tau \in [\tau_0, \tau_1].
\end{equation}
From \eqref{bd:estimationtau} we infer that for $\kappa<\frac{d-2}{4}$, we have $\nu^2\leq e^{-2\kappa \tau}$ for $\tau_0$ large enough depending on $\bar M_0$. Lemma \ref{lemm:mqinn}, together with this inequality and \eqref{bd:boundarynorm} then implies
\be \label{est:fs}
 \frac{d}{ds} (e^{\delta_2 s}f) \leq e^{\delta_2 s}(C\nu^{-2}e^{\frac A 2} \|  m_\varepsilon (\tau) \|_{\bd}^2 + C\nu^2)
\leq e^{\delta_2 s-2\kappa \tau}  \left(CK^\frac{5}{2}e^{-\frac A4} +C\right).
\ee
Recall that $\frac{d\tau}{ds} =\nu$ so that $\frac{d}{ds}(\delta_2 s-2\kappa \tau)\geq \frac{\delta_2}{2}$ for $\tau_0$ large enough. From this, we deduce that $e^{\delta_2(s_0-s)}\leq e^{2\kappa (\tau_0-\tau)}$ and $\int_{s_0}^s e^{\delta_2 \tilde s-2\kappa \tau(\tilde s)}d\tilde s \leq \frac{2}{\delta_2}e^{\delta_2 s-2\kappa \tau} $. Integrating \eqref{est:fs} with time $s$ using these two inequalities yields
\begin{align}
\nonumber f(s) & \leq e^{\delta_2 (s_0-s)} f(s_0)+e^{-\delta_2 s}  \left(CK^\frac{5}{2}e^{-\frac A4} +C\right) \int_{s_0}^s e^{\delta_2 \tilde s-2\kappa \tau(\tilde s)}  d\tilde s \\
\label{proofbootstrap:inter1}& \leq e^{2 \kappa (\tau_0-\tau)}f(s_0)+e^{-2\kappa \tau}\left(CK^\frac{5}{2}e^{-\frac A4} +C\right) \ \leq \ e^{-2\kappa \tau}\left(CK^\frac{5}{2}e^{-\frac A4} +C\right),
\end{align}
where we used \eqref{est:mqL2ome} with constant $K=1$ at initial time $s_0$ from \eqref{bd:bootstrap:initial}. The estimate \eqref{proofbootstrap:inter1} implies \eqref{est:ftaus} upon choosing $K$ large enough with $K e^{-\frac  A2}$ small enough. Hence \eqref{proofbt:1}, \eqref{proofbt:2}, \eqref{proofbt:3} and \eqref{proofbt:4} are valid.

Let now $\mathcal T$ be the set of times $\tau_1\geq \tau_0$ such that the solution is trapped on $[\tau_0,\tau_1]$. By continuity, the set $\mathcal T$ is closed. Now, for any $\tau_1\in \mathcal T$, the inequalities underlying Definition \ref{def:bootstrap} are strict inequalities at time $s_1$ as they are improved by the factor $\frac{1}{2}$ using \eqref{proofbt:1}, \eqref{proofbt:2}, \eqref{proofbt:3} and \eqref{proofbt:4}. Hence by continuity of the flow of \eqref{eq:murt}, we have $[\min (\tau_1-\delta,\tau_0),\tau_1+\delta]\subset \mathcal T$ for some $\delta$ small enough, so that $\mathcal T$ is open in $[s_0,\infty)$. By connectedness, $\mathcal T=[s_0,\infty)$ which concludes the proof of Proposition \ref{prop:trap_mq}.  \end{proof}

\begin{proof}[Proof of Theorem \ref{theo:1}] Theorem \ref{theo:1} is just a direct consequence of Proposition \ref{prop:trap_mq}. Recall that $\frac{d\tau}{dt} = \frac{M(t)}{R^d(t)}$, we use \eqref{modulation:bd:asymptoticM} and \eqref{modulation:bd:asymptoticR} to write
$$\frac{d \tau}{dt} = \frac{M_\infty}{\tilde{R}_\infty^d  } e^{\frac{d}{2}\tau} \big[1 + \Oc(e^{-\kappa \tau})\big].$$
Solving this equation yields the existence of $T > 0$ such that
\begin{equation}\label{eq:reltauT}
\tau =   - \frac{2}{d} \log \left( \frac{d M_\infty}{2 \tilde{R}_\infty^d } (T-t) \right)\big[1 + o_{t \to T}(1)\big].
\end{equation}
Hence, the estimate \eqref{modulation:bd:asymptoticR} is written in terms of the $t$ variable as
$$
R(t) = \tilde{R}_\infty e^{ \frac{1}{2} \Big[\frac{2}{d} \log \left( \frac{d M_\infty}{2 \tilde{R}_\infty^d } (T-t) \right)\Big] } \big[1 + o_{t \to T}(1)\big] = \big[\frac d2 M_\infty (T-t)\big]^\frac{1}{d}\big[1 + o_{t \to T}(1)\big]. 
$$
Unwinding the change of variables \eqref{def:murt}, \eqref{def:mw}, \eqref{def:mepmq}, one gets
$$
u(r)=\frac{M}{R^{d-1}\lambda}(\pa_\xi Q(\xi)+\tilde u(r)), \qquad \tilde u(r)=\frac{1}{\zeta^{d-1}} \pa_\xi \tilde Q(\xi)-\pa_\xi Q(\xi)+\frac{1}{\zeta^{d-1}}\pa_\xi m_\varepsilon (\zeta),
$$
where $\tilde Q(\xi)=\bar Q_\nu(\zeta)$. Since the solution is global in time $\tau$, the desired estimate \eqref{eq:bdvarepsilonmainth} for $\tilde u$ then directly follows from \eqref{def:Qxi}, \eqref{def:xis} and \eqref{modulation:bd:asymptoticnu} to estimate the first term, and \eqref{est:mepLinf_out}, \eqref{est:mepLinf_out2} and \eqref{est:mqL2ome} to estimate the second one (upon using a parabolic regularity argument for $\xi_{-}\leq \xi \leq \xi_+$ similar to Lemma \ref{lem:rightboundary} that we omit).

We now turn to the continuity of the blowup dynamics. Fix $u_0$ satisfy the requirements of Proposition \ref{prop:trap_mq}, and let $u$ solve \eqref{3Dsys} with data $u_0$. Then any $v_0$ close enough to $u_0$ in $L^\infty$ satisfies the requirements of Proposition \ref{prop:trap_mq} with same bootstrap constants ($A,K,...$), so that the solution $v$ to \eqref{3Dsys} with data $v_0$ blows up at time $T(v_0)$ and satisfies \eqref{eq:decomp_murt}, \eqref{eq:bdparametersmainth} and \eqref{eq:bdvarepsilonmainth} as well. We now prove the continuity of $T$ and $M_\infty$. Let $\tau_u,M_u,R_u$ and $\tau_v,M_v,R_v$ denote $u$ and $v$ related parameters respectively. Integrating the relations $dt/d\tau=R^d/M$, and using \eqref{modulation:bd:asymptoticR} and \eqref{modulation:bd:asymptoticM} we obtain:
\begin{align}
\label{end:id:parametersu} T(u_0)=\int_{\tau_u(0)}^\infty \frac{R^d_u(\tau)}{M_u(\tau)}d\tau, \qquad |M_u(\tau)-M_\infty(u_0)|+\int_{\tau}^\infty \left|\frac{R^d_u(\tilde \tau)}{M_u(\tilde \tau)}\right|d\tilde \tau \leq Ce^{-\kappa \tau} \quad \forall \tau\geq \tau_u(0),\\
\label{end:id:parametersv} T(v_0)=\int_{\tau_v(0)}^\infty \frac{R^d_v(\tau)}{M_v(\tau)}d\tau, \qquad |M_v(\tau)-M_\infty(v_0)|+\int_{\tau}^\infty \left|\frac{R^d_v(\tilde \tau)}{M_v(\tilde \tau)}\right|d\tilde \tau\leq Ce^{-\kappa \tau} \quad \forall \tau\geq \tau_v(0),
\end{align}
with same constant $C>0$. Let now $\delta>0$. There then exists $t_\delta\in [0,T(u_0))$ such that $Ce^{-\kappa \tau_u(t_\delta)}\leq \delta/4$. By continuity of the flow of \eqref{3Dsys} with respect to the initial data in $L^\infty(\mathbb R^d)$ (see \cite{B} and references therein), we have that $T(v_0)\geq t_\delta$ for $v_0$ close to $u_0$, and that $\tau_v\to \tau_u$, $R_v\to R_u$ and $M_v\to M_u$ uniformly on $[0,t_\delta]$, as $v_0\to u_0$. Hence for $v_0$ close enough to $u_0$, $Ce^{-\kappa \tau_v(t_\delta)}\leq \delta/3$, $|M_v(t_\delta)-M_u(t_\delta)|\leq \delta/3$, and $|\int_{\tau_v(0)}^{\tau_v(t_\delta)} \frac{R^d_v(\tau)}{M_v(\tau)}d\tau- \int_{\tau_u(0)}^{\tau_u(t_d)} \frac{R^d_u(\tau)}{M_u(\tau)}d\tau|\leq \delta/3$. Combining these inequalities with \eqref{end:id:parametersu} and \eqref{end:id:parametersv} one obtains that $|T(v_0)-T(u_0)|\leq \delta$ and $|M_\infty(v_0)-M_\infty(u_0)|\leq \delta$. This proves the continuity of $T$ and $M_\infty$ and ends the proof of Theorem \ref{theo:1}.
\end{proof}

\appendix 

\section{Functional analysis} \label{ap:1}

\begin{lemma}[Poincar\'e and Sobolev in $H^1_{\omega_0}$.] \label{lemm:Poincare}
There exists $C>0$ such that the following inequalities hold true for any $u\in H^1_{\omega_0}$,
\begin{align}\label{iq:Poincare}
& \int_{\mathbb R} |u(y)|^2 \omega_0 (y)dy\leq C\int_{\mathbb R} |\partial_y u(y)|^2\omega_0 (y)dy,\\
\label{iq:LinfSobolev}
& |u(\xi)| \leq Ce^{-\frac{|\xi|}{4}}  \left(\int_{\mathbb R} |\partial_y u(y)|^2 \omega_0 (y)dy\right)^{\frac 12} \qquad \mbox{for all }\xi\in \mathbb R.
\end{align}
\end{lemma}
\begin{proof}

The second inequality \eqref{iq:LinfSobolev} is a direct consequence of the fundamental Theorem of Calculus and of Cauchy-Schwarz. Indeed, as $\omega_0 u^2\in L^1(\mathbb R)$, there exists $y_n\to -\infty$ such that $\omega_0(y_n) u^2(y_n)\to 0$ and hence $u(y_n)\to 0$. For $y\leq 0$, we have $u(y)=u(y_n)+\int_{y_n}^y \pa_y u$ so that $u(y)=\int_{0}^y \pa_y u$ by letting $n\to \infty$. We estimate since $\omega_0\approx e^{|\xi|/2}$:
$$
|u(y)|=|\int_{-\infty}^y \pa_y u |\leq \left(\int_{-\infty}^y|\partial_y u|^2\omega_0 \right)^{\frac 12}  \left(\int_{-\infty}^{y} \omega_0^{-1} \right)^{\frac 12} \lesssim  \left(\int_{-\infty}^y|\partial_y u|^2\omega_0 \right)^{\frac 12} e^{\frac y4}.
$$
For $y\geq 0$ the proof is the same upon replacing $-\infty$ by $\infty$ in the integrals. Hence \eqref{iq:LinfSobolev} is proved. From \eqref{iq:LinfSobolev}, we deduce
\be \label{poincare:bd:inter1}
\int_{-2}^2 u^2dy \lesssim \int_{\mathbb R} |\partial_y u(y)|^2 \omega_0 (y)dy.
\ee
Take now $\chi$ a cut-off function such that $\chi(y)=2$ for $y\geq 1$, $\chi(y)=0$ for $y\leq 1$, and write $u_1(y)=\chi(y)u$ and $u_2(y)=\chi(-y)u$. We have for $i=1,2$, integrating by parts:
$$
\int_{\mathbb R} u_i \pa_y u_i \omega_0 dy=-\frac{1}{2} \int_{\mathbb R} u_i^2 \pa_y \omega_0dy.
$$
Since on $(-\infty,1]$ and $[1,\infty)$ we have $\pa_y\omega_0\approx \omega_0$ from the formula \eqref{def:omega}, we deduce that
$$
\int_{\mathbb R} |u_i|^2\omega_0 \lesssim \left| \int_{\mathbb R} u_i \pa_y u_i \omega_0\right| \qquad \mbox{for }i=1,2.
$$
Applying Cauchy-Schwarz and Young inequalities yields
\be \label{poincare:bd:inter2}
\int_{\mathbb R}|u_i|^2\omega_0 \lesssim \int_{\mathbb R}|\pa_y u_i|^2\omega_0 \lesssim \int_{\mathbb R}|\pa_y u|^2\omega_0,
\ee
where we used \eqref{iq:LinfSobolev} in the last inequality. Combining \eqref{poincare:bd:inter1} and \eqref{poincare:bd:inter2} shows \eqref{iq:Poincare}. \end{proof}

\begin{lemma}[Coercivity of $\Ls_0$] \label{lemm:coerL02} There exists $A^*,\delta_1>0$ such that the following holds true for all $A\geq A^*$. Assume that $f\in H^2_{\omega_0}$ satisfies $\int_{\mathbb R} f \pa_\xi Q \chi_A \omega_0 d\xi = 0$. Then:
\begin{align}
 \label{coer:id:coercivity2}
& \delta_1  \| f\|^2_{H^1_{\omega_0}} \leq - \int_{\mathbb R} \Ls_0 f f \omega_0 d\xi, \\
 \label{coer:id:coercivity}
& \delta_1  \| f\|^2_{H^2_{\omega_0}} \leq \int_{\mathbb R} |\Ls_0 f|^2 \omega_0 d\xi.
\end{align}
\end{lemma}

\begin{proof} We first decompose
\be
\label{coer:id:decomposition}
f=c\pa_\xi Q+g, \qquad \mbox{with}\quad \int_{\mathbb R} g \pa_\xi Q \omega_0=0.
\ee
We compute by integrating \eqref{coer:id:decomposition} against $\pa_\xi Q\omega_0$ that $c=\| \pa_\xi Q\|_{L^2_{\omega_0}}^{-2}\int_{\mathbb R} f(1-\chi_A)\pa_\xi Q\omega_0d\xi$. Using Cauchy-Schwarz, $\omega_0(\xi)\approx e^{|\xi|/2}$ and $|\pa_\xi Q(\xi)|\lesssim e^{-|\xi|/2}$ we get $|c|\lesssim e^{-A/4}\| f\|_{L^2_{\omega_0}}$. Thus:
\begin{align}
\label{coer:bd:preliminary2}
& \| f\|_{H^1_{\omega_0}}\leq 2 \| g\|_{H^1_{\omega_0}} \quad \mbox{and} \quad \int_{\mathbb R} \mathcal L_0 f f \omega_0 d\xi=  \int_{\mathbb R} \mathcal L_0 gg \omega_0d\xi,\\
\label{coer:bd:preliminary}
& \| f\|_{H^2_{\omega_0}}\leq 2 \| g\|_{H^2_{\omega_0}} \quad \mbox{and} \quad \| \mathcal L_0 f\|_{L^2_{\omega_0}} = \| \mathcal L_0 g\|_{L^2_{\omega_0}} ,
\end{align}
 for $A$ large enough for the inequalities, and using $\mathcal L_0\pa_\xi Q=0$ for the equalities.
 
 We now apply Lemma \ref{lemm:coerLs0} to $g$ and get $\langle -\mathcal L_0 g, g\rangle_{L^2_{\omega_0}}\geq \delta \| g \|_{H^1_{\omega_0}}^2$. This inequality and \eqref{coer:bd:preliminary2} imply the first estimate \eqref{coer:id:coercivity2} of the Lemma. Using $|xy|\leq \delta |x|/2+|y|/2\delta$ we have $\langle -\mathcal L_0 g, g\rangle_{L^2_{\omega_0}}\leq \| \mathcal L_0 g \|_{L^2_{\omega_0}}^2/2\delta+\delta \| g\|_{L^2_{\omega_0}}^2/2$. Combining these two inequalities yields
\be
\label{coer:bd:inter1}
\| g\|_{H^1_{\omega_0}}\leq \delta^{-1} \| \mathcal L_0 g\|_{L^2_{\omega_0}}.
\ee
Since $\mathcal L_0g=\pa_\xi^2 g-(1/2-Q)\pa_\xi g+\pa_\xi Q g$, we deduce $|\pa_\xi^2 g|\leq |\mathcal L_0 g|+|\pa_\xi g|/2+|g|/2$, so that:
\be
\label{coer:bd:inter2}
\| \pa_\xi^2 g\|_{L^2_{\omega_0}}\leq \| \mathcal L_0 g\|_{L^2_{\omega_0}}+\| g\|_{H^1_{\omega_0}}.
\ee
Combining \eqref{coer:bd:inter1} and \eqref{coer:bd:inter2} shows $\| g\|_{H^2_{\omega_0}}\leq C(\delta)\| \mathcal L_0 g\|_{L^2_\omega}$. Combining this inequality and \eqref{coer:bd:preliminary} shows the second inequality of the Lemma \eqref{coer:id:coercivity}.\end{proof}

\def\cprime{$'$}

\end{document}